\documentclass[12pt,twoside]{amsart}
\usepackage{amssymb,amsmath,amsthm, amscd, enumerate, mathrsfs}
\usepackage{graphicx, hhline}
\usepackage[all]{xy}
\usepackage[usenames]{color}
\usepackage{hyperref}
\usepackage{fancyhdr}
\hypersetup{colorlinks=true}

\title{Minimal model theory for relatively trivial log canonical pairs}
\author{Kenta Hashizume}
\date{2017/11/20, version 0.08}
\address{Department of Mathematics, Graduate School of Science, 
Kyoto University, Kyoto 606-8502, Japan}
\email{hkenta@math.kyoto-u.ac.jp}

\newtheorem{theo}{Theorem}[section]
\newtheorem{lemm}[theo]{Lemma}
\newtheorem{coro}[theo]{Corollary}
\newtheorem{prop}[theo]{Proposition}
\newtheorem{conj}[theo]{Conjecture}

\theoremstyle{definition}
\newtheorem{defi}[theo]{Definition}
\newtheorem{rema}[theo]{Remark}
\newtheorem*{ack}{Acknowledgments} 

\newtheorem*{divisor}{Divisors} 
\newtheorem*{sing}{Singularities of pairs} 
\newtheorem{step}{Step}
\newtheorem{step2}{Step}
\newtheorem{step3}{Step}
\newtheorem{case}{Case}

\begin{document}

\maketitle
\begin{abstract}
We study relative log canonical pairs with relatively trivial log canonical divisors. 
We fix such a pair $(X,\Delta)/Z$ and establish the minimal model theory for the pair $(X,\Delta)$ assuming the minimal model theory for all Kawamata log terminal pairs whose dimension is not greater than ${\rm dim}\,Z$. 
We also show the finite generation of log canonical rings for log canonical pairs of dimension five which are not of log general type.
\end{abstract}

\tableofcontents

\section{Introduction}\label{sec1}
Throughout this paper we work over $\mathbb{C}$, the complex number field. 

In the minimal model theory for higher-dimensional algebraic varieties, one of the most important problems is the existence of a good minimal model or a Mori fiber space for log pairs. 
In this paper we only deal with the case when the boundary divisor is a $\mathbb{Q}$-divisor.

\begin{conj}\label{conjminimalmodel}
Let $(X,\Delta)$ be a projective log canonical pair such that $\Delta$ is a $\mathbb{Q}$-divisor. 
If $K_{X}+\Delta$ is pseudo-effective then $(X,\Delta)$ has a good minimal model, and if $K_{X}+\Delta$ is not pseudo-effective then $(X,\Delta)$ has a Mori fiber space. 
\end{conj}

Conjecture \ref{conjminimalmodel} for log canonical threefolds is proved by Kawamata, Koll\'ar, Matsuki, Mori, Shokurov and others. 
Conjecture \ref{conjminimalmodel} for Kawamata log terminal pairs with big boundary divisors is also proved by Birkar, Cascini, Hacon and M\textsuperscript{c}Kernan \cite{bchm}. 
But Conjecture \ref{conjminimalmodel} is still open when the dimension is greater than three. 

An interesting case of Conjecture \ref{conjminimalmodel} is when $(X,\Delta)$ is a relative log canonical pair whose log canonical divisor is relatively trivial. 
The situation is a special case of lc-trivial fibration, which is expected to play a crucial role in inductive arguments. 
For example, Ambro's canonical bundle formula for Kawamata log terminal pairs, which is proved by Ambro \cite{ambro}, gives an inductive argument. 
On any klt-trivial fibration $(X,\Delta)\to Z$, which is a special case of lc-trivial fibration, Conjecture \ref{conjminimalmodel} for $(X,\Delta)$ can be reduced to Conjecture \ref{conjminimalmodel} for a Kawamata log terminal pair on $Z$ by the canonical bundle formula.  
Ambro's canonical bundle formula is expected to hold for log canonical pairs in full generality but it is only partially solved (cf.~\cite{fujino-remarks}, \cite{fg-bundle} and \cite{fg-lctrivial}).

In this paper we establish an inductive argument for log canonical pairs in the above situation.
The following is the main result of this paper. 

\begin{theo}\label{thmmain}
Fix a nonnegative integer $d_{0}$, and assume the existence of a good minimal model or a Mori fiber space for all $d$-dimensional  projective Kawamata log terminal pairs with boundary $\mathbb{Q}$-divisors such that $d\leq d_{0}$. 

Let $\pi:X \to Z$ be a projective surjective morphism of normal projective varieties such that ${\rm dim}\,Z\leq d_{0}$, 
and let $(X,\Delta)$ be a log canonical pair such that $\Delta$ is a $\mathbb{Q}$-divisor. 
Suppose that $K_{X}+\Delta \sim_{\mathbb{Q},\,Z}0$.  

Then $(X,\Delta)$ has a good minimal model or a Mori fiber space.
\end{theo}

The following theorem follows from Theorem \ref{thmmain}. 

\begin{theo}\label{thmtrueiitaka}
Let $(X,\Delta)$ be a projective log canonical pair such that $\Delta$ is a $\mathbb{Q}$-divisor and the log Kodaira dimension $\kappa(X,K_{X}+\Delta)$ is nonnegative. 
Let $F$ be the general fiber of the Iitaka fibration and $(F,\Delta_{F})$ be the restriction of $(X,\Delta)$ to $F$. 
Suppose that $(F,\Delta_{F})$ has a good minimal model. 

If $(X,\Delta)$ is Kawamata log terminal or $\kappa(X,K_{X}+\Delta)\leq 4$, then $(X,\Delta)$ has a good minimal model. 
\end{theo}

As a corollary of Theorem \ref{thmtrueiitaka}, we have 

\begin{coro}\label{corlcring}
Let $(X,\Delta)$ be a projective log canonical pair such that ${\rm dim}\,X =5$ and $\Delta$ is a $\mathbb{Q}$-divisor. 
If $(X,\Delta)$ is not of log general type, then the log canonical ring $\mathcal{R}(X,K_{X}+\Delta)$ is a finitely generated $\mathbb{C}$-algebra. 
\end{coro}

We recall some previous results related to Theorem \ref{thmmain}, Theorem \ref{thmtrueiitaka} and Corollary \ref{corlcring}. 
In \cite{gongyolehmann}, Gongyo and Lehmann established an inductive argument for $\mathbb{Q}$-factorial Kawamata log terminal pairs $(X,\Delta)$ with a contraction $f:X\to Z$ such that $\nu((K_{X}+\Delta)|_{F})=0$, where $\nu(\,\cdot\,)$ is the numerical dimension and $F$ is the general fiber of $f$. 
More precisely, in the situation, they proved existence of a Kawamata log terminal pair $(Z',\Delta_{Z'})$ such that $Z'$ is birational to $Z$ and $(X,\Delta)$ has a good minimal model if and only if $(Z',\Delta_{Z'})$ has a good minimal model.    
On the other hand, Birkar and Hu \cite{birkarhu-arg} proved existence of a good minimal model for log canonical pairs $(X,\Delta)$ when $K_{X}+\Delta$ is the pullback of a big divisor on a normal variety whose augmented base locus does not contain the image of any lc center of $(X,\Delta)$.  
In particular they proved existence of a good minimal model for all log canonical pairs $(X,\Delta)$ when $K_{X}+\Delta$ is big and its augmented base locus does not contain any lc center of $(X,\Delta)$. 
On the other hand, Lai \cite{lai} proved Theorem \ref{thmtrueiitaka} in the case when $X$ has at worst terminal singularities and $\Delta=0$. 
Related to Corollary \ref{corlcring}, Fujino \cite{fujino-lcring} proved the finite generation of log canonical rings for all log canonical fourfolds. 
In the klt case, Birkar, Cascini, Hacon and M\textsuperscript{c}Kernan \cite{bchm} proved the finite generation of log canonical rings in all dimensions. 

In Theorem \ref{thmmain}, the case when $X=Z$ implies equivalence of Conjecture \ref{conjminimalmodel} for Kawamata log terminal pairs and Conjecture \ref{conjminimalmodel} for log canonical pairs (see also \cite{fg-lcring}). 
If $(X,\Delta)$ is Kawamata log terminal, then Theorem \ref{thmmain} follows from Ambro's canonical bundle formula for Kawamata log terminal pairs (cf.~Proposition \ref{propkltcase}).
But since we do not assume that $(X,\Delta)$ is Kawamata log terminal, we can not use the canonical bundle formula directly. 
We also note that we do not have any assumptions about lc centers of $(X,\Delta)$ in Theorem \ref{thmmain}.
We hope that Theorem \ref{thmmain} will play an important role in inductive arguments for the minimal model program.

We outline the proof of Theorem \ref{thmmain}. 
We prove Theorem \ref{thmmain} by induction on $d_{0}$. 
Note that we can assume $K_{X}+\Delta$ is pseudo-effective. 
First we take a special dlt blow-up of $(X,\Delta)$ so that $\Delta=\Delta'+\Delta''$ where $\Delta''$ is a reduced divisor or $\Delta''=0$ and all lc centers of $(X,\Delta')$ dominate $Z$. 
Next we replace $\pi:(X,\Delta)\to Z$ so that $Z$ is $\mathbb{Q}$-factorial and $K_{X}+\Delta'\sim_{\mathbb{Q},\,Z}0$. 
Then we can apply Ambro's canonical bundle formula to $(X,\Delta')$, and if we write $K_{X}+\Delta\sim_{\mathbb{Q}}\pi^{*}D$, we can run the $D$-MMP with scaling. 
When $\Delta''=0$ (in particular when $(X,\Delta)$ is klt), a good minimal model for $D$ exists and we can show the existence of a good minimal model of $(X,\Delta)$ by using it. 
For details, see Section \ref{secreduction}. 
When $\Delta\neq0$, we divide into three cases. 
In any case, we can check that for any sufficiently small rational number $u>0$ the pair $(X,\Delta-u\Delta'')$ has a good minimal model or a Mori fiber space. 
For example, assume $(X,\Delta-u\Delta'')$ has a good minimal model and $D$ is not big, where $D$ satisfies $K_{X}+\Delta\sim_{\mathbb{Q}}\pi^{*}D$. 
This situation is one of the three cases and other cases are proved similarly. 
By choosing $u$ sufficiently small, we can construct a modification $(X',\Delta_{X'}) \to Z'$ of $(X,\Delta)\to Z$ such that $X\dashrightarrow X'$ is a sequence of finitely many steps of the $(K_{X}+\Delta-u\Delta'')$-MMP to a good minimal model. 
We can also check that Theorem \ref{thmmain} for $(X',\Delta_{X'})\to Z'$ implies Theorem \ref{thmmain} for $(X,\Delta)\to Z$. 
Since $D$ is not big, the contraction $X'\to W$ induced by $K_{X'}+\Delta_{X'}-u\Delta''_{X'}$ satisfies ${\rm dim}\,W<{\rm dim}\,Z$. 
Furthermore, by choosing $u$ appropriately, we can assume $X'\to W$ satisfies $K_{X'}+\Delta_{X'}\sim_{\mathbb{Q},\,W}0$. 
Then $(X',\Delta_{X'})$ has a good minimal model by the induction hypothesis. 
Thus we see that $(X,\Delta)$ has a good minimal model. 
For details, see Section \ref{secproofmainthm}. 

The contents of this paper are the following. 
In Section \ref{secpre} we collect some definitions and notations. 
In section \ref{secmmp} we introduce the definition of $D$-MMP, where $D$ is $\mathbb{R}$-Cartier and not necessarily a log canonical divisor, and prove some related result. 
In Section \ref{secreduction} we prove Theorem \ref{thmmain} in a spacial case which contains the klt case, and reduce Theorem \ref{thmmain} to a special situation. 
In Section \ref{secproofmainthm} we prove Theorem \ref{thmmain}. 
In Section \ref{secproofcorollary}, we prove Theorem \ref{thmtrueiitaka} and Corollary \ref{corlcring}. 

\begin{ack}
The author was partially supported by JSPS KAKENHI Grant Number JP16J05875.
The author would like to express his gratitude to his supervisor Professor Osamu Fujino for useful advice. 
He thanks Professor Paolo Cascini for a comment about Theorem \ref{thmtrueiitaka}. 
He also thanks Takahiro Shibata for discussions. 
He is grateful to the referee for valuable comments and suggestions. 
\end{ack}

\section{Preliminaries}\label{secpre}
In this section we collect some notations and definitions. 
We will freely use the notations and definitions in \cite{kollar-mori} and \cite{bchm} except the definition of models (see Definition \ref{defnmodels}). 
Here we write down only some important notations and definitions, including the notations not written in \cite{kollar-mori} or \cite{bchm}.

\begin{divisor}\label{saydivisors}
Let $\pi:X \to Z$ be a projective morphism of normal varieties and let $D=\sum d_{i}D_{i}$ be a $\mathbb{Q}$-divisor. 
Then $D$ is a {\it boundary $\mathbb{Q}$-divisor} if $0\leq d_{i} \leq 1$ for any $i$. 
The {\it round down} of $D$, denoted by $\llcorner D \lrcorner$, is $\sum \llcorner d_{i} \lrcorner D_{i}$ 
where $\llcorner d_{i} \lrcorner$ is the largest integer which is not greater than $d_{i}$. 
Suppose that $D$ is $\mathbb{Q}$-Cartier. 
Then $D$ is called a {\it log canonical divisor} if $D$ is the sum of the canonical divisor $K_{X}$ and a boundary $\mathbb{Q}$-divisor. 
$D$ is {\it trivial over $Z$}, denoted by $D\sim_{\mathbb{Q},\,Z}0$, if $D$ is $\mathbb{Q}$-linearly equivalent to the pullback of a  $\mathbb{Q}$-Cartier $\mathbb{Q}$-divisor on $Z$. 
$D$ is {\it anti-ample over $Z$} if $-D$ is ample over $Z$. 
In this paper we mean the same definition by saying that $D$ is {\it trivial} (resp.~{\it anti-ample}) {\it with respect to $\pi$}. 
$D$ is {\it semi-ample over} $Z$ if $D$ is a $\mathbb{Q}_{\geq 0}$-linear combination of semi-ample Cartier divisors over $Z$, or equivalently, there exists a morphism $f: X \to Y$ to a variety $Y$ over $Z$ such that $D$ is $\mathbb{Q}$-linearly equivalent to the pullback of an ample $\mathbb{Q}$-divisor over $Z$. 

For any $\mathbb{Q}$-divisor $D$ on $X$, we define a sheaf of $\mathcal{O}_{Z}$-algebra
$$\mathcal{R}(X/Z,D)=\underset{m\geq 0}{\bigoplus}\pi_{*}\mathcal{O}_{X}(\llcorner mD \lrcorner).$$
We simply denote $\mathcal{R}(X,D)$ when $Z$ is a point. 
If $D$ is a log canonical divisor, then $\mathcal{R}(X/Z,D)$ is nothing but the log canonical ring.

Similarly we can define boundary divisors, log canonical divisors, triviality over $Z$, semi-ampleness over $Z$, and so on for $\mathbb{R}$-divisors. 

Let $X \dashrightarrow Y$ be a birational map of normal projective varieties and let $D$ be an $\mathbb{R}$-divisor on $X$. 
Unless otherwise stated, we mean the birational transform of $D$ on $Y$ by denoting $D_{Y}$ or $(D)_{Y}$. 
\end{divisor}

\begin{sing}\label{saypairs}
Let $X$ be a normal variety and $\Delta$ be an effective $\mathbb{R}$-divisor such that $K_X+\Delta$ is $\mathbb{R}$-Cartier. 
Let $f:Y\to X$ be a log resolution of $(X, \Delta)$. 
Then we can write 
$$K_Y=f^*(K_X+\Delta)+\sum_{i} a(E_{i}, X,\Delta) E_i$$ 
where $E_{i}$ are prime divisors on $Y$ and $a(E_{i}, X,\Delta)$ is a real number for any $i$. 
Then we call $a(E_{i}, X,\Delta)$ the {\it discrepancy} of $E_{i}$ with respect to $(X,\Delta)$. 
The pair $(X, \Delta)$ is called {\it Kawamata log terminal} ({\it klt}, for short) if $a(E_{i}, X, \Delta) > -1$ for any log resolution $f$ of $(X, \Delta)$ and any $E_{i}$ on $Y$. 
$(X, \Delta)$ is called {\it log canonical} ({\it lc}, for short) if $a(E_{i}, X, \Delta) \geq -1$ for any log resolution $f$ of $(X, \Delta)$ and any $E_{i}$ on $Y$. 
$(X, \Delta)$ is called {\it divisorial log terminal} ({\it dlt}, for short) if $\Delta$ is a boundary $\mathbb{R}$-divisor and there exists a log resolution $f:Y \to X$ of $(X, \Delta)$ such that $a(E, X, \Delta) > -1$ for any $f$-exceptional prime divisor $E$ on $Y$.
\end{sing}

Next we recall the construction of dlt models. 
The following theorem is proved by Hacon.

\begin{theo}[Dlt blow-ups, cf.~{\cite[Theorem 10.4]{fujino-fund}}, {\cite[Theorem 3.1]{kollarkovacs}}]\label{thmdltblowup}
Let $X$ be a normal quasi-projective variety of dimension $n$ 
and let $\Delta$ be an $\mathbb{R}$-divisor such that $(X, \Delta)$ is log canonical. 
Then there exists a projective birational morphism $f:Y \to X$ from a normal quasi-projective 
variety $Y$ such that 
\begin{itemize}
\item[(i)]
$Y$ is $\mathbb{Q}$-factorial, and
\item[(ii)]
if we set
$$\Gamma=f_{*}^{-1}\Delta +\sum_{E{\rm :}f \mathchar`- {\rm exceptional}}E,$$
then $(Y, \Gamma)$ is dlt and $K_{Y}+\Gamma=f^{*}(K_{X}+\Delta)$.
\end{itemize}
We call $(Y,\Gamma)$ a {\it dlt model} of $(X,\Delta)$. 
\end{theo}

Next we introduce the definition of some models and the construction of the log MMP with scaling for $\mathbb{Q}$-factorial log canonical pairs. 
Our definition of models is slightly different from the traditional one in \cite{kollar-mori} or \cite{bchm}.  

\begin{defi}[cf.~{\cite[Definition 2.1]{birkar-flip}} and {\cite[Definition 2.2]{birkar-flip}}]\label{defnmodels}
Let $\pi:X \to Z$ be a projective morphism from a normal variety to a variety and let $(X,\Delta)$ be a log canonical pair.  
Let $\pi ':X' \to Z$ be a projective morphism from a normal variety to $Z$ and $\phi:X \dashrightarrow X'$ be a birational map over $Z$. 
Let $E$ be the reduced $\phi^{-1}$-exceptional divisor on $X'$, that is, $E=\sum E_{j}$ where $E_{j}$ are $\phi^{-1}$-exceptional prime divisors on $X'$. 
Then the pair $(X', \Delta'=\phi_{*}\Delta+E)$ is called a {\it log birational model} of $(X,\Delta)$ over $Z$. 
A log birational model $(X', \Delta')$ of $(X,\Delta)$ over $Z$ is a {\it weak log canonical model} ({\it weak lc model}, for short) if 
\begin{itemize}
\item
$K_{X'}+\Delta'$ is nef over $Z$, and 
\item
for any prime divisor $D$ on $X$ which is exceptional over $X'$, we have
$$a(D, X, \Delta) \leq a(D, X', \Delta').$$ 
\end{itemize}
A weak lc model $(X',\Delta')$ of $(X,\Delta)$ over $Z$ is a {\it log minimal model} if 
\begin{itemize}
\item
$X'$ is $\mathbb{Q}$-factorial, and 
\item
the above inequality on discrepancies is strict. 
\end{itemize}
A log minimal model $(X',\Delta')$ of $(X, \Delta)$ over $Z$ is called a {\it good minimal model} if $K_{X'}+\Delta'$ is semi-ample over $Z$.

A log birational model $(X',\Delta')$ of $(X, \Delta)$ over $Z$ is called a {\it Mori fiber space} if $X'$ is $\mathbb{Q}$-factorial and there is a contraction $X' \to W$ over $Z$ with ${\rm dim}\,W<{\rm dim}\,X'$ such that 
\begin{itemize}
\item
the relative Picard number $\rho(X'/W)$ is one and $K_{X'}+\Delta'$ is anti-ample over $W$, and 
\item
for any prime divisor $D$ over $X$, we have
$$a(D,X,\Delta)\leq a(D,X',\Delta')$$
and strict inequality holds if $D$ is a divisor on $X$ and exceptional over $X'$.
\end{itemize} 
\end{defi}

\begin{defi}[The log MMP with scaling, cf.~{\cite[Definition 2.4]{birkar-flip}}, {\cite[4.4.11]{fujino-book}}]\label{defnmmpwithscaling}
Let $\pi:X \to Z$ be a projective surjective morphism from a $\mathbb{Q}$-factorial normal variety to a variety and $(X,\Delta+C)$ be a log canonical pair such that $K_{X}+\Delta+C$ is $\pi$-nef, $\Delta$ is a boundary $\mathbb{R}$-divisor and $C$ is an effective $\mathbb{R}$-divisor. 
We set $X_{0}=X$, $\Delta_{X_{0}}=\Delta$ and $C_{X_{0}}=C$ and set
$$\lambda_{0}={\rm inf}\{\mu \in \mathbb{R}_{\geq0}\,|\,K_{X_{0}}+\Delta_{X_{0}}+\mu C_{X_{0}}{\rm \;is\;nef \; over\;}Z\}.$$
If $\lambda_{0}=0$, we have nothing to do. 
If $\lambda_{0}>0$, then there is a $(K_{X_{0}}+\Delta_{X_{0}})$-negative extremal ray $R_{0}$ over $Z$ such that $(K_{X_{0}}+\Delta_{X_{0}}+\lambda_{0}C_{X_{0}})\cdot R_{0} =0$ by \cite[Theorem 18.9]{fujino-fund}. 
Let $f_{0}:X_{0}\to V_{0}$ be the extremal contraction over $Z$ given by $R_{0}$. 
If ${\rm dim}\,V_{0}<{\rm dim}\,X_{0}$, then we stop. 
Assume ${\rm dim}\,V_{0}={\rm dim}\,X_{0}$. 
Then $f_{0}$ is birational. 
If $f_{0}$ is a divisorial contraction, then set $X_{1}=V_{0}$, $\Delta_{X_{1}}=f_{0*}\Delta_{X_{0}}$ and $C_{X_{1}}=f_{0*}C_{X_{0}}$. 
If $f_{0}$ is a flipping contraction, then there is the flip $\phi:X_{0}\dashrightarrow X_{1}$ of $f_{0}$ over $Z$ by \cite[Corollary 1.2]{birkar-flip} or \cite[Corollary 1.8]{haconxu-lcc}, and we set $\Delta_{X_{1}}=\phi_{*}\Delta_{X_{0}}$ and $C_{X_{1}}=\phi_{*}C_{X_{0}}$. 
By construction $X_{1}$ is $\mathbb{Q}$-factorial.
We set 
$$\lambda_{1}={\rm inf}\{\mu \in \mathbb{R}_{\geq0}\,|\,K_{X_{1}}+\Delta_{X_{1}}+\mu C_{X_{1}}{\rm \;is\;nef \; over\;}Z\}.$$
Then we have $\lambda_{1} \leq \lambda_{0}$. 
If $\lambda_{1}=0$, we stop the process. 
If $\lambda_{1}>0$, then there is a $(K_{X_{1}}+\Delta_{X_{1}})$-negative extremal ray $R_{1}$ over $Z$ such that $(K_{X_{1}}+\Delta_{X_{1}}+\lambda_{1}C_{X_{1}})\cdot R_{1} =0$. 
By repeating this process, we get a non-increasing sequence of nonnegative real numbers $\{\lambda_{i}\}_{i \geq0}$ and a sequence of steps of the $(K_{X}+\Delta)$-MMP over $Z$
$$(X=X_{0},\Delta=\Delta_{X_{0}})\dashrightarrow(X_{1},\Delta_{X_{1}})\dashrightarrow \cdots \dashrightarrow(X_{i},\Delta_{X_{i}})\dashrightarrow\cdots .$$
This log MMP is called the {\it $(K_{X}+\Delta)$-MMP over $Z$ with scaling of $C$}.
\end{defi}

\begin{rema}\label{remdefnmodels}
Let $(X,\Delta)$ be a log canonical pair and $(X',\Delta')$ be a log minimal model or a Mori fiber space of $(X,\Delta)$.
If the birational map $X\dashrightarrow X'$ is a birational contraction, our definition of log minimal models and Mori fiber spaces coincides with the traditional one. 

In \cite{birkar-flip}, log minimal models and Mori fiber spaces are supposed to be dlt. 
On the other hand we do not assume it in Definition \ref{defnmodels}. 
But the difference is intrinsically not important. 
Indeed, if a log canonical pair $(X,\Delta)$ has a log minimal model $(X',\Delta')$ as in Definition \ref{defnmodels}, any dlt model of $(X',\Delta')$ is also a log minimal model of $(X,\Delta)$. 
If $(X,\Delta)$ has a Mori fiber space as in Definition \ref{defnmodels}, we can construct a Mori fiber space of $(X,\Delta)$ which is dlt by taking a dlt model of $(X,\Delta)$ and by running the log MMP with scaling. 
In this way, when $(X,\Delta)$ has a log minimal model (resp.~Mori fiber space), we can construct a log minimal model (resp.~Mori fiber space) of $(X,\Delta)$ which is dlt. 
\end{rema}

Next we introduce the definition of log canonical thresholds and pseudo-effective thresholds.

\begin{defi}[Log canonical thresholds, {cf.~\cite{hmx-acc}}]\label{defnlct}
Let $(X,\Delta)$ be a log canonical pair and let $M\neq0$ be an effective $\mathbb{R}$-Cartier $\mathbb{R}$-divisor. 
Then the {\it log canonical threshold} of $M$ with respect to $(X,\Delta)$, denoted by ${\rm lct}(X,\Delta;M)$, is 
$${\rm lct}(X,\Delta;M)={\rm sup}\{t\in \mathbb{R}\,|\, (X,\Delta+tM) {\rm \; is \; log \;canonical}\}.$$ 
\end{defi}

\begin{defi}[Pseudo-effective thresholds]
Let $(X,\Delta)$ be a projective log canonical pair and $M$ be an effective $\mathbb{R}$-Cartier $\mathbb{R}$-divisor such that $K_{X}+\Delta+tM$ is pseudo-effective for some $t \geq 0$. 
Then the {\it pseudo-effective threshold} of $M$ with respect to $(X,\Delta)$, denoted by $\tau(X,\Delta;M)$, is 
$$\tau(X,\Delta;M)={\rm inf}\{t\in \mathbb{R}_{\geq 0}\,|\, K_{X}+\Delta+tM {\rm \; is \; pseudo\mathchar`-effective}\}.$$ 
\end{defi}

The following important theorems are proved by Hacon, M\textsuperscript{c}Kernan and Xu \cite{hmx-acc}. 
In their paper, one is called the ACC for log canonical thresholds and another one is called the ACC for numerically trivial pairs. 

\begin{theo}[ACC for log canonical thresholds, {cf.~\cite[Theorem 1.1]{hmx-acc}}]\label{thmacclct}
Fix a positive integer $n$, a set $I \subset [0,1]$ and a set $J\subset \mathbb{R}_{>0}$, where $I$ and $J$ satisfy the DCC. 
Let $\mathfrak{T}_{n}(I)$ be the set of log canonical pairs $(X,\Delta)$, where $X$ is a variety of dimension $n$ and the coefficients of $\Delta$ belong to $I$. 
Then the set 
$$\{ {\rm lct}(X,\Delta;M) \,|\, (X,\Delta) \in \mathfrak{T}_{n}(I), {\rm \; the \;coefficients \; of\; }M{\rm \;belong\; to\;}J \}$$
satisfies the ACC.
\end{theo}

\begin{theo}[ACC for numerically trivial pairs, {cf.~~\cite[Theorem 1.5]{hmx-acc}}]\label{thmglobalacc}
Fix a positive integer $n$ and a set $I \subset [0,1]$, which satisfies the DCC. 

Then there is a finite set $I_{0}\subset I$ with the following property:

If $(X,\Delta)$ is a log canonical pair such that 
\begin{enumerate}
\item[(i)]
X is projective of dimension $n$, 
\item[(ii)]
the coefficients of $\Delta$ belong to $I$, and
\item[(iii)]
$K_{X}+\Delta$ is numerically trivial,
\end{enumerate}
then the coefficients of $\Delta$ belong to $I_{0}$. 
\end{theo}

Finally we introduce the definition of log smooth models and two related results. 
Corollary \ref{cordltblowup} is a special kind of dlt blow-up used in this paper.

\begin{defi}[Log smooth models, cf.~{\cite[Definition 2.3]{birkar-flip}} and {\cite[Remark 2.8]{birkar-flip}}]\label{defnlogsmoothmodel}
Let $(X,\Delta)$ be a log canonical pair and $f:Y \to X$ be a log resolution of $(X,\Delta)$. 
Let $\Gamma$ be a boundary $\mathbb{R}$-divisor on $Y$ such that $(Y,\Gamma)$ is log smooth. 
Then $(Y,\Gamma)$ is a {\it log smooth model} of $(X,\Delta)$ if we write 
$$K_{Y}+\Gamma=f^{*}(K_{X}+\Delta)+F, $$
then
\begin{enumerate}
\item[(i)]
$F$ is an effective $f$-exceptional divisor, and 
\item[(ii)] 
every $f$-exceptional prime divisor $E$ satisfying $a(E,X,\Delta)>-1$ is a component of $F$ and $\Gamma-\llcorner \Gamma \lrcorner$.  
\end{enumerate}
By the definition, ${\rm Supp}\,\Gamma ={\rm Supp}\,f_{*}^{-1}\Delta \cup {\rm Ex}\,(f)$ and the image of any lc center of $(Y,\Gamma)$ on $X$ is an lc center of $(X,\Delta)$. 
For any $f$-exceptional prime divisor $E$, $E$ is a component of $F$ if and only if $a(E,X,\Delta)>-1$. 
When $\Delta$ is a $\mathbb{Q}$-divisor and $f:Y \to X$ is a log resolution of $(X,\Delta)$, we can find a $\mathbb{Q}$-divisor $\Gamma$ on $Y$ such that $(Y,\Gamma)$ is a log smooth model of $(X,\Delta)$. 
\end{defi}

\begin{lemm}\label{lemlogresol}
Let $\pi:X \to Z$ be a projective morphism from a normal variety to a variety. 
Let $(X,\Delta)$ be a log canonical pair. 
Then there is a log smooth model $(Y,\Gamma)$ of $(X,\Delta)$ such that 
\begin{enumerate}
\item[(i)]
$\Gamma=\Gamma'+\Gamma''$, where $\Gamma'\geq0$ and $\Gamma''$ is a reduced divisor,
\item[(ii)]
$(\pi \circ f)({\rm Supp}\,\Gamma'') \subsetneq Z$, and 
\item[(iii)]
every lc center of $(Y,\Gamma-t\Gamma'')$ dominates $Z$ for any $0< t \leq 1$.
\end{enumerate} 
\end{lemm}

\begin{proof}
Replacing $(X,\Delta)$ with its log smooth model, we can assume that $(X,\Delta)$ is log smooth. 
For any lc center $S$ of $(X,\Delta)$ not dominating $Z$, let $\pi_{S}:X_{S} \to X$ be the blow-up of $X$ along $S$.  
Then $X_{S}$ is a smooth variety and $\pi_{S}^{-1}(S)$ is a divisor on $X_{S}$. 
In particular it is a Cartier divisor on $X_{S}$. 
Let $f:Y\to X$ be a log resolution of $(X,\Delta)$ such that $Y$ is also a common resolution of all $X_{S}$, and construct a log smooth model $(Y,\Gamma)$ of $(X,\Delta)$. 
Let $\Gamma''$ be the reduced divisor such that $\Gamma''$ is the sum of all components of $\llcorner \Gamma \lrcorner$ not dominating $Z$, and set $\Gamma'=\Gamma-\Gamma''$.   
Then $\Gamma'$ and $\Gamma''$ satisfy the conditions (i) and (ii) of the lemma. 
We prove that $\Gamma'$ and $\Gamma''$ satisfy the condition (iii) of the lemma.

Fix $0<t\leq 1$ and let $T$ be an lc center of $(Y,\Gamma-t\Gamma'')$. 
Since $(Y,\Gamma)$ is lc, $T$ is an lc center of $(Y,\Gamma)$ and $T$ is not contained in ${\rm Supp}\, \Gamma''$. 
We prove that $T$ dominates $Z$.

Suppose by contradiction that $T$ does not dominate $Z$. 
Then $f(T)$ is an lc center of $(X,\Delta)$ not dominating $Z$ and therefore $\pi_{f(T)}^{-1}(f(T))$ is a Cartier divisor on $X_{f(T)}$ by construction. 
Set $M=f^{-1}(f(T))$. 
Clearly we have $(\pi \circ f)(M)\subsetneq Z$, and $M$ is a divisor because $M$ is the support of the pullback of $\pi_{f(T)}^{-1}(f(T))$. 
Moreover $T$ is contained in a component of $M$ because $T$ is irreducible. 
Since ${\rm Supp}\,\Gamma={\rm Supp}\,f_{*}^{-1}\Delta \, \cup \, {\rm Ex}\,(f)$, we also have $M \subset {\rm Supp}\,\Gamma$. 
Therefore $T$ is contained in a component $G$ of $\Gamma$ such that $(\pi \circ f)(G)\subsetneq Z$.
On the other hand, $T$ is an irreducible component of the intersection of some divisors in 
$\llcorner \Gamma \lrcorner$ because $T$ is an lc center of the log smooth model $(Y,\Gamma)$. 
Since $(Y,\Gamma)$ is log smooth, the coefficient of $G$ in $\Gamma$ is one. 
Then $T$ is contained in ${\rm Supp}\,\Gamma''$ and we get a contradiction. 
Therefore $T$ dominates $Z$ and so we are done. 
\end{proof}

Let $\pi:X \to Z$ be a projective morphism from a normal variety to a variety and $(X,\Delta)$ be a log canonical pair. 
In the rest of this paper, the phrase \lq\lq $(X,\Delta=\Delta'+\Delta'')$ satisfies all the conditions of Lemma \ref{lemlogresol} with respect to $\pi$\rq\rq \;means that we can write $\Delta=\Delta'+\Delta''$ where $\Delta'$ and $\Delta''$ satisfy the conditions (i), (ii) and (iii) of Lemma \ref{lemlogresol} with respect to $\pi$.

\begin{coro}\label{cordltblowup}
Let $\pi:X \to Z$ be a projective morphism of normal quasi-projective varieties and $(X,\Delta)$ be a log canonical pair. 
Then there is a dlt blow-up $f:(Y,\Gamma) \to (X,\Delta)$ such that $(Y,\Gamma=\Gamma'+\Gamma'')$ satisfies all the conditions of Lemma \ref{lemlogresol} with respect to $\pi\circ f$. 
\end{coro}

\begin{proof}
Let $(Y,\Gamma=\Gamma'+\Gamma'')\to (X,\Delta)$ be a log smooth model of $(X,\Delta)$ as in Lemma \ref{lemlogresol}. 
We run the $(K_{Y}+\Gamma)$-MMP over $X$ with scaling. 
By \cite[Theorem 3.4]{birkar-flip}, we get a good minimal model $\phi:(Y,\Gamma)\dashrightarrow(Y',\Gamma_{Y'})$ over $X$. 
Let $f:Y' \to X$ be the induced morphism. 
Then $f$ is a dlt blow-up of $(X,\Delta)$. 
Set $\Gamma_{Y'}'=\phi_{*}\Gamma'$ and $\Gamma_{Y'}''=\phi_{*}\Gamma''$. 
Then we can check that $(Y,  \Gamma_{Y'}=\Gamma_{Y'}'+\Gamma_{Y'}'')$ satisfies all the conditions of Lemma \ref{lemlogresol} with respect to $\pi \circ f$ because $a(D,Y,\Gamma-t\Gamma'')\leq a(D,Y',\Gamma_{Y'}-t\Gamma''_{Y'})$ for any sufficiently small positive real number $t$ and any prime divisor $D$ over $Y$. 
Therefore $f:(Y',\Gamma_{Y'}) \to (X,\Delta)$ is the desired dlt blow-up. 
\end{proof}

\section{Minimal model program}\label{secmmp}

In this section we study the Minimal Model Program for any $\mathbb{R}$-Cartier $\mathbb{R}$-divisor $D$ which is not necessarily a log canonical divisor. 
More precisely, we define a sequence of birational maps, which we call $D$-MMP, and construct the $D$-MMP under some assumptions. 

\begin{defi}[The $D$-MMP]\label{defnd-mmp}
Let $X$ be a $\mathbb{Q}$-factorial normal projective variety and let $D$ be an $\mathbb{R}$-divisor on $X$.
Then a finite sequence of birational maps
$$\phi:X=X_{0}\dashrightarrow X_{1}\dashrightarrow \cdots \dashrightarrow X_{n}$$
is a sequence of finitely many steps of the $D${\it -Minimal Model Program} ($D${\it -MMP}, for short) if 
\begin{enumerate}
\item[(i)]
there exists a boundary $\mathbb{R}$-divisor $\Delta$ on $X$ such that $(X,\Delta)$ is log canonical and $\phi$ is a sequence of finitely many steps of the $(K_{X}+\Delta)$-MMP, and 
\item[(ii)]
for any $0\leq i<n$, the birational transform $D_{X_{i}}$ of $D$ on $X_{i}$, which is always $\mathbb{R}$-Cartier by the condition (i), is anti-ample with respect to the extremal contraction $f_{i}:X_{i} \to V_{i}$, that is, $X_{i+1}=V_{i}$ or $X_{i+1}$ is the flip of $f_{i}$. 
\end{enumerate} 
An infinite sequence of birational maps
$$X=X_{0}\dashrightarrow X_{1}\dashrightarrow \cdots \dashrightarrow X_{i}\dashrightarrow \cdots$$
is a sequence of steps of the $D${\it -MMP} if $X\dashrightarrow X_{i}$ is a sequence of finitely many steps of the $D$-MMP for any $i$.
\end{defi}

\begin{rema}\label{remdefnd-mmp}
Our definition of $D$-MMP is slightly different from usual one because we assume (i), that is, any sequence of finitely many steps of the $D$-MMP is always the log MMP for a log canonical divisor.

By the definition, all $X_{i}$ are $\mathbb{Q}$-factorial. 
In (i) of the above definition, we can in fact find a boundary $\mathbb{Q}$-divisor instead of a boundary $\mathbb{R}$-divisor. 

Notation as above, suppose that $D$ is a $\mathbb{Q}$-divisor and let $X_{i}\dashrightarrow X_{i+1}$ be a step of the $D$-MMP. 
Then it is a step of the $(K_{X}+\Delta)$-MMP for some $\mathbb{Q}$-divisor $\Delta$. 
Let $X_{i}\to V_{i}$ be the extremal contraction. 
Then we can write $X_{i+1}={\boldsymbol{\rm Proj}}\bigl(\mathcal{R}(X_{i}/V_{i},D_{X_{i}})\bigr)$. 
Indeed, we can check that $X_{i+1}={\boldsymbol{\rm Proj}}\bigl(\mathcal{R}(X_{i}/V_{i},K_{X_{i}}+\Delta_{X_{i}})\bigr)$ even if $X_{i} \to V_{i}$ is a divisorial contraction. 
By the cone theorem \cite[Theorem 4.5.2]{fujino-book} and since $D_{X_{i}}$ is anti-ample over $V_{i}$, $D_{X_{i}} \sim_{\mathbb{Q},\,V_{i}}m(K_{X_{i}}+\Delta_{X_{i}})$ for some positive rational number $m$. 
Thus $X_{i+1}\simeq{\boldsymbol{\rm Proj}}\bigl(\mathcal{R}(X_{i}/V_{i},D_{X_{i}})\bigr)$. 
\end{rema}

\begin{defi}[The $D$-MMP with scaling]\label{defnd-mmpwithscaling}
Let $X$ be a $\mathbb{Q}$-factorial normal projective variety and let $D$ be an $\mathbb{R}$-divisor on $X$.
Let $A$ be an $\mathbb{R}$-divisor such that $D+A$ is nef. 
Then a sequence of birational maps 
$$X=X_{0}\dashrightarrow X_{1}\dashrightarrow \cdots \dashrightarrow X_{i}\dashrightarrow\cdots$$
is the $D${\it -MMP with scaling of $A$} if 
\begin{enumerate}
\item[(i)]
it is a sequence of steps of the $D$-MMP, and
\item[(ii)]
if we set 
$$\lambda_{i}={\rm inf}\{\mu \in \mathbb{R}_{\geq0} \, | \, D_{X_{i}}+\mu A_{X_{i}}{\rm \;is \;nef}\}$$ 
for any $i$, then $D_{X_{i}}+\lambda_{i}A_{X_{i}}$ is trivial with respect to the extremal contraction $X_{i}\to V_{i}$. 
\end{enumerate} 
If divisors $D$ and $A$ on $X$ are given and there is no confusion, we denote the $D$-MMP with scaling of $A$ by 
$$(X=X_{0},\lambda_{0})\dashrightarrow (X_{1},\lambda_{1})\dashrightarrow \cdots \dashrightarrow (X_{i},\lambda_{i})\dashrightarrow\cdots$$ 
where $\lambda_{i}={\rm inf}\{\mu \in \mathbb{R}_{\geq0} \,|\, D_{X_{i}}+\mu A_{X_{i}}{\rm \;is \;nef}\}$. 
\end{defi}

\begin{rema}\label{remd-mmp}
Notation as above, let 
$$(X=X_{0},\lambda_{0})\dashrightarrow (X_{1},\lambda_{1})\dashrightarrow \cdots \dashrightarrow (X_{i},\lambda_{i})\dashrightarrow\cdots$$ 
be a sequence of steps of the $D$-MMP with scaling of $A$. 
Pick an index $i\geq0$ and a real number $t<\lambda_{i}$, which is not necessarily positive.
By the definition of the $D$-MMP with scaling, the sequence of birational maps $X_{0}\dashrightarrow \cdots \dashrightarrow X_{i}\dashrightarrow X_{i+1}$ is a sequence of  finitely many steps of the $(D+tA)$-MMP with scaling of $(1-t)A$. 
If we set $\lambda'_{j}=(\lambda_{j}-t)/(1-t)$ for any $0 \leq j \leq i$, then the $(D+tA)$-MMP with scaling can be written
$$(X_{0},\lambda'_{0})\dashrightarrow  \cdots \dashrightarrow (X_{i},\lambda'_{i})\dashrightarrow X_{i+1}.$$ 
In particular, if $t<\lambda_{i}$ for any $i$, then the above sequence of birational maps is the $(D+tA)$-MMP with scaling of $(1-t)A$ 
$$(X_{0},\lambda'_{0})\dashrightarrow  \cdots \dashrightarrow (X_{i},\lambda'_{i})\dashrightarrow \cdots$$
where $\lambda'_{i}=(\lambda_{i}-t)/(1-t)$ for any $i$.  
\end{rema}

If $D$ is the log canonical divisor of a log canonical pair, we can identify the $D$-MMP with the standard log MMP on the log canonical pair. 
Therefore Definition \ref{defnd-mmp} is a generalization of the standard log MMP. 
Similarly, we can check that Definition \ref{defnd-mmpwithscaling} is a generalization of the standard definition of the log MMP with scaling.

Finally, we prove two results related to the $D$-MMP with scaling.

\begin{lemm}\label{lemd-mmp}
Let $\pi:X\to Z$ be a projective surjective  morphism from a normal projective variety to a $\mathbb{Q}$-factorial normal  projective variety with connected fibers, and let $(X,\Delta)$ be a log canonical pair such that $\Delta$ is a $\mathbb{Q}$-divisor. 
Suppose that $(X,\Delta=\Delta'+\Delta'')$ satisfies all the conditions of Lemma \ref{lemlogresol} with respect to $\pi$.
Suppose in addition that $K_{X}+\Delta \sim_{\mathbb{Q}}\pi^{*}D$ and $\Delta'' \sim_{\mathbb{Q}}\pi^{*}E$ for a $\mathbb{Q}$-divisor $D$ and an effective $\mathbb{Q}$-divisor $E$ on $Z$ respectively. 
Let $A$ be a big semi-ample $\mathbb{Q}$-divisor on $Z$ such that $A+E$ is also semi-ample and $D+A$ is nef.

Then there is a sequence of birational  maps of the $D$-MMP with scaling of $A$ 
$$(Z=Z_{0},\lambda_{0}) \dashrightarrow \cdots \dashrightarrow (Z_{i},\lambda_{i})\dashrightarrow \cdots$$
such that the $D$-MMP terminates or ${\rm lim}_{i \to \infty} \lambda_{i}=0$ when the $D$-MMP does not terminate. 
In particular, we always have ${\rm lim}_{i \to \infty} \lambda_{i}=0$ when $D$ is pseudo-effective.
\end{lemm}

\begin{lemm}\label{lemliftmmp}
Let $\pi:X\to Z$ be a projective surjective morphism of $\mathbb{Q}$-factorial normal projective varieties with connected fibers, and let $(X,\Delta)$ be a log canonical pair where $\Delta$ is a $\mathbb{Q}$-divisor. 
Suppose that $(X,0)$ is Kawamata log terminal and there is a $\mathbb{Q}$-divisor $D$ on $Z$ such that $K_{X}+\Delta \sim_{\mathbb{Q}}\pi^{*}D$. 
Let $A$ be an effective $\mathbb{Q}$-divisor on $Z$ such  that $D+A$ is nef and $(X,\Delta+\pi^{*}A)$ is log canonical. 
Suppose that there is a sequence of birational maps of the $D$-MMP with scaling of $A$
$$(Z=Z_{0},\lambda_{0}) \dashrightarrow \cdots \dashrightarrow (Z_{i},\lambda_{i})\dashrightarrow \cdots$$
with the corresponding numbers $\lambda_{i}$ defined in Definition \ref{defnd-mmpwithscaling}. 
We set $X_{0}=X$ and $\Delta_{X_{0}}=\Delta$. 

Then we have the following diagram 
$$
\xymatrix@C=21pt{
(X_{0},\Delta_{X_{0}}) \ar_{\pi=\pi_{0}}[d] \ar@{-->}[r]&
\cdots \ar@{-->}[r]&(X_{k_{1}},\Delta_{X_{k_{1}}})\ar_{\pi_{1}}[d] \ar@{-->}[r]&
\cdots\ar@{-->}[r]& (X_{k_{i}},\Delta_{X_{k_{i}}}) \ar_{\pi_{i}}[d] \ar@{-->}[r]&
\cdots \\
(Z_{0}, \lambda_{0}) \ar@{-->}[rr]&&
(Z_{1}, \lambda_{1}) \ar@{-->}[r]&
\cdots\ar@{-->}[r]&(Z_{i}, \lambda_{i})\ar@{-->}[r]&
\cdots
}
$$
such that 
\begin{enumerate}
\item[(i)]
for any $i$, $\pi_{i}$ is projective and surjective with connected fibers, 
\item[(ii)]
the upper horizontal sequence of birational maps is a sequence of steps of the $(K_{X}+\Delta)$-MMP with scaling of $\pi^{*}A$ such that 
if we set $k_{0}=0$ and 
$$\lambda'_{j}={\rm inf}\{\mu \in \mathbb{R}_{\geq0} \,|\, K_{X_{j}}+\Delta_{X_{j}}+\mu (\pi^{*}A)_{X_{j}}{\rm \;is \;nef}\},$$
where $(\pi^{*}A)_{X_{j}}$ is the birational transform of $\pi^{*}A$ on $X_{j}$, then $\lambda'_{j}=\lambda_{i}$ for any $k_{i}\leq j<k_{i+1}$, and
\item[(iii)]
for any two indices $i<i'$ and any $\mathbb{Q}$-divisor $B$ on $Z_{i}$, we have $(\pi_{i}^{*}B)_{X_{k_{i'}}}=\pi_{i'}^{*}B_{Z_{i'}}$.
\end{enumerate}
In particular, $K_{X_{k_{i}}}+\Delta_{X_{k_{i}}}\sim_{\mathbb{Q}}\pi_{i}^{*}D_{Z_{i}}$ for any $i$ and the $(K_{X}+\Delta)$-MMP with scaling of $\pi^{*}A$ terminates if and only if the $D$-MMP with scaling of $A$ terminates. 
\end{lemm}

\begin{proof}[Proof of Lemma \ref{lemd-mmp}]
Fix a strictly decreasing infinite sequence of rational numbers $\{a_{n}\}_{n\geq1}$ such that $0<a_{n}<1$ for any $n$ and ${\rm lim}_{n \to \infty}a_{n}=0$.
By the condition (iii) of Lemma \ref{lemlogresol}, any lc center of $(X,\Delta-a_{n}\Delta'')$ dominates $Z$ for any $n$. 
We also have $K_{X}+\Delta-a_{n}\Delta'' \sim_{\mathbb{Q}}\pi^{*}(D-a_{n}E)$ by the hypothesis. 
By \cite[Corollary 3.2]{fg-bundle}, there are $\mathbb{Q}$-divisors $\Psi_{n}$ on $Z$ such that all  $(Z,\Psi_{n})$ are klt and $D-a_{n}E \sim_{\mathbb{Q}}K_{Z}+\Psi_{n}$. 
Fix a sufficiently general semi-ample $\mathbb{Q}$-divisor $A' \sim_{\mathbb{Q}} A+E$ such that $(Z,\Psi_{n}+A')$ is klt for any $n\geq1$. 
Similarly, fix a sufficiently general semi-ample $\mathbb{Q}$-divisor $A'' \sim_{\mathbb{Q}} A$ such that $(Z,\Psi_{n}+A'+A'')$ is klt for any $n\geq1$.  
Then $(Z, \Psi_{n}+a_{n}A'+tA'')$ is also klt and 
\begin{equation*}\tag{$\spadesuit$}
\begin{split}D+(t+a_{n})A&=(D-a_{n}E)+a_{n}(A+E)+tA\\
&\sim_{\mathbb{Q}}K_{Z}+\Psi_{n}+a_{n}A'+tA''
\end{split}
\end{equation*} 
for any $0\leq t\leq1$ and $n\geq1$.
We note that $A'$ is big. 

Since $K_{Z}+\Psi_{1}+a_{1}A'+(1-a_{1})A''\sim_{\mathbb{Q}}D+A$ is nef, we can run the $(K_{Z}+\Psi_{1}+a_{1}A')$-MMP with scaling of $(1-a_{1})A''$. 
By \cite[Corollary 1.4.2]{bchm}, this log MMP terminates with a good minimal model or a Mori fiber space 
$$\phi:Z=Z_{0}\dashrightarrow Z_{1}\dashrightarrow \cdots \dashrightarrow Z_{k_{1}}=Z'$$
of $(Z,\Psi_{1}+a_{1}A')$. 
It is also a sequence of finitely many steps of the $(D+a_{1}A)$-MMP since we have  $D+a_{1}A\sim_{\mathbb{Q}}K_{Z}+\Psi_{1}+a_{1}A'$. 

For any $i\geq0$, we set
$$\lambda_{i}={\rm inf}\{\mu \in \mathbb{R}_{\geq0} \,|\, K_{Z_{i}}+(\Psi_{1})_{Z_{i}}+a_{1}A'_{Z_{i}}+\mu (1-a_{1})A''_{Z_{i}}{\rm \;is \;nef}\}$$ 
where $(\Psi_{1})_{Z_{i}}$ is the birational transform of $\Psi_{1}$ on $Z_{i}$. 
We note that $\lambda_{k_{1}-1}>0$ by the definition of the log MMP with scaling. 
By the above $(\spadesuit)$, for any $0\leq i <k_{1}$, $D_{Z_{i}}+\bigl(a_{1}+\lambda_{i}(1-a_{1})\bigr)A_{Z_{i}}$ is nef and trivial with respect to the extremal contraction of the  $(K_{Z}+\Psi_{1}+a_{1}A')$-MMP. 
Since $D_{Z_{i}}+a_{1}A_{Z_{i}}$ is anti-ample with respect to the extremal contraction, $D_{Z_{i}}$ is anti-ample with respect to the extremal contraction for any $0\leq i <k_{1}$. 
Moreover, if we set 
$$\lambda'_{i}={\rm inf}\{\mu \in \mathbb{R}_{\geq0} \,|\, D_{Z_{i}}+\mu A_{Z_{i}}{\rm \;is \;nef}\}$$ 
for any $0\leq i<k_{1}$, then $\lambda'_{i}=a_{1}+\lambda_{i}(1-a_{1})$ by the above discussion. 
Therefore $\phi$ is a sequence of finitely many steps of the $D$-MMP with scaling of $A$ (see Definition \ref{defnd-mmpwithscaling}).
Pick a rational number $t \leq a_{1}$. 
Then we have $t<\lambda'_{k_{1}-1}$ since $a_{1}<\lambda'_{k_{1}-1}$.
By Remark \ref{remd-mmp}, $\phi$ is a sequence of finitely many steps of the $(D+tA)$-MMP with scaling of $(1-t)A$ for any $0\leq t\leq a_{1}$. 
Since $K_{Z}+\Psi_{n}+a_{n}A'\sim_{\mathbb{Q}}D+a_{n}A$, $A'' \sim_{\mathbb{Q}}A$ and $\{a_{n}\}_{n\geq 1}$ is a strictly decreasing sequence, we see that $\phi$ is also a sequence of finitely many steps of the $(K_{Z}+\Psi_{n}+a_{n}A')$-MMP with scaling of $(1-a_{n})A''$ for any $n\geq1$. 

If $\bigl(Z', (\Psi_{1})_{Z'}+a_{1}A'_{Z'}\bigr)$ is a Mori fiber space, then the $D$-MMP with scaling terminates and we stop the process. 
If $\bigl(Z', (\Psi_{1})_{Z'}+a_{1}A'_{Z'}\bigr)$ is a good minimal model of $(Z, \Psi_{1}+a_{1}A')$, then $\lambda_{k_{1}}=0$.
By the above $(\spadesuit)$ we have
\begin{equation*}
\begin{split}K_{Z'}+(\Psi_{1})_{Z'}+a_{1}A'_{Z'}&\sim_{\mathbb{Q}}D_{Z'}+a_{1}A_{Z'}\\
&\sim_{\mathbb{Q}}K_{Z'}+(\Psi_{2})_{Z'}+a_{2}A'_{Z'}+(a_{1}-a_{2})A''_{Z'},
\end{split}
\end{equation*} 
and thus $K_{Z'}+(\Psi_{2})_{Z'}+a_{2}A'_{Z'}+(a_{1}-a_{2})A''_{Z'}$ is nef. 
Moreovrer the pair $\bigl(Z',(\Psi_{2})_{Z'}+a_{2}A'_{Z'}\bigr)$ is klt since $\phi$ is a sequence of finitely many steps of the $(K_{Z}+\Psi_{2}+a_{2}A')$-MMP. 
So we can run the $\bigl(K_{Z'}+(\Psi_{2})_{Z'}+a_{2}A'_{Z'}\bigr)$-MMP with scaling $(a_{1}-a_{2})A''_{Z'}$.
By \cite[Corollary 1.4.2]{bchm}, this log MMP terminates with a good minimal model or a Mori fiber space
$$\psi:Z'=Z_{k_{1}}\dashrightarrow Z_{k_{1}+1}\dashrightarrow \cdots \dashrightarrow Z_{k_{2}}=Z''$$
of $\bigl(Z',(\Psi_{2})_{Z'}+a_{2}A'_{Z'}\bigr)$. 
By the same discussion as above, we can check that $\psi\circ \phi:Z \dashrightarrow Z''$ is a sequence of finitely many steps of the $D$-MMP with scaling of $A$ and also a sequence of finitely many steps of the $(K_{Z}+\Psi_{n}+a_{n}A')$-MMP with scaling of $(1-a_{n})A''$ for any $n\geq2$.

By repeating the above discussions, we get a sequence of birational maps 
$$Z=Z_{0}\dashrightarrow Z_{1}\dashrightarrow \cdots\dashrightarrow Z_{k_{i}}\dashrightarrow \cdots$$
such that
\begin{itemize}
\item
for any $i\geq1$, the birational map $Z\dashrightarrow Z_{k_{i}}$ is a sequence of finitely many steps of the $(K_{Z}+\Psi_{i}+a_{i}A')$-MMP with scaling of $(1-a_{i})A''$ to a good minimal model or a Mori fiber space,
\item
the whole sequence of birational maps $Z\dashrightarrow \cdots\dashrightarrow Z_{j}\dashrightarrow \cdots$ is a sequence of steps of the $D$-MMP with scaling of $A$, and
\item
if the $D$-MMP does not terminate and if we set 
$$\lambda_{j}={\rm inf}\{\mu \in \mathbb{R}_{\geq0} \,|\, D_{Z_{j}}+\mu A_{Z_{j}}{\rm \;is \;nef}\},$$
then $\lambda_{k_{i}}\leq a_{i}$.
\end{itemize}
The third condition follows from the fact that $D_{Z_{k_{i}}}+a_{i}A_{Z_{k_{i}}}$ is nef. 
By the definition of $\{a_{n}\}_{n\geq1}$, we have ${\rm lim}_{i\to \infty}\lambda_{i}\leq{\rm lim}_{i\to \infty}a_{i}=0$ when the $D$-MMP does not terminate. 
Therefore ${\rm lim}_{i\to \infty}\lambda_{i}=0$ and hence we see that the above $D$-MMP with scaling of $A$ satisfies all the conditions of the lemma. 
So we are done.
\end{proof}

\begin{proof}[Poof of Lemma \ref{lemliftmmp}]
Set $\pi_{0}=\pi$ and let $f:Z_{0}\to V_{0}$ be the extremal contraction. 
Note that $K_{X_{0}}+\Delta_{X_{0}}+\pi_{0}^{*}A$ is nef since $D+A$ is nef. 
By the definition of $D$-MMP and the cone theorem \cite[Theorem 4.5.2]{fujino-book}, there is a general ample $\mathbb{Q}$-divisor $H$ on $Z_{0}$ such that $D+H\sim_{\mathbb{Q},\,V_{0}}0$. 
Therefore $K_{X_{0}}+\Delta_{X_{0}}+\pi_{0}^{*}H\sim_{\mathbb{Q},\,V_{0}}0$. 
By \cite[Theorem 1.1]{birkar-flip} and \cite[Theorem 4.1 (iii)]{birkar-flip}, the $(K_{X_{0}}+\Delta_{X_{0}})$-MMP over $V_{0}$ with scaling of an ample divisor terminates with a good minimal model 
$$\phi:(X_{0},\Delta_{X_{0}})\dashrightarrow(X_{1},\Delta_{X_{1}})\dashrightarrow \cdots \dashrightarrow (X_{k_{1}}=X',\Delta_{X_{k_{1}}}=\Delta_{X'})$$ 
over $V_{0}$. 
Then we can check that $\phi$ is a sequence of finitely many steps of the $(K_{X_{0}}+\Delta_{X_{0}})$-MMP with scaling of $\pi_{0}^{*}A$ and if we set 
$$\lambda'_{j}={\rm inf}\{\mu \in \mathbb{R}_{\geq0} \,|\, K_{X_{j}}+\Delta_{X_{j}}+\mu (\pi_{0}^{*}A)_{X_{j}}{\rm \;is \;nef}\}$$ 
for any $0\leq j\leq k_{1}$, then $\lambda'_{0}=\lambda'_{1}=\cdots = \lambda'_{k_{1}-1}=\lambda_{0}$ (see, for example, the proof of \cite[Proposition 4.1]{has-ab}). 

Since $K_{X'}+\Delta_{X'}$ is semi-ample over $V_{0}$, there is a natural morphism $X'\to Z'={\boldsymbol{\rm Proj}}\bigl(\mathcal{R}(X'/V_{0},K_{X'}+\Delta_{X'})\bigr)$ over $V_{0}$.
By construction we have $K_{X'}+\Delta_{X'}\sim_{\mathbb{Q},\,Z'}0$. 
Now $Z_{1}={\boldsymbol{\rm Proj}}\bigl(\mathcal{R}(Z_{0}/V_{0},D)\bigr)$ by Remark \ref{remdefnd-mmp}, and for any large and divisible positive integer $m$, we have 
\begin{equation*}
\begin{split}
\mathcal{R}(Z_{0}/V_{0},mD) &\simeq\mathcal{R}\bigl(X_{0}/V_{0},m(K_{X_{0}}+\Delta_{X_{0}})\bigr)\\
&\simeq \mathcal{R}\bigl(X'/V_{0},m(K_{X'}+\Delta_{X'})\bigr)
\end{split}
\end{equation*}
as sheaves of graded $\mathcal{O}_{V_{0}}$-algebra. 
Therefore we have $Z'\simeq Z_{1}$.
We put $\pi_{1}:X' \to Z_{1}\simeq Z'$. 
Then we see that $\pi_{1}$ has connected fibers by taking a common resolution of $\phi$.
We also see that $K_{X'}+\Delta_{X'} \sim_{\mathbb{Q}} \pi_{1}^{*} D_{Z_{1}}$ because $K_{X'}+\Delta_{X'}\sim_{\mathbb{Q},\,Z_{1}}0$ and $K_{X_{0}}+\Delta_{X_{0}}\sim_{\mathbb{Q}}\pi_{0}^{*}D$. 
Since $X_{0}$ is $\mathbb{Q}$-factorial and $(X_{0},0)$ is klt, it is easy to see that $X'$ is $\mathbb{Q}$-factorial and $(X',0)$ is klt. 

We prove $(\pi_{0}^{*}B)_{X'}=\pi_{1}^{*}B_{Z_{1}}$ for any $\mathbb{Q}$-divisor $B$ on $Z_{0}$. 
First we prove $(\pi_{0}^{*}B)_{X'}\sim_{\mathbb{Q}}\pi_{1}^{*}B_{Z_{1}}$, and after that  
we prove $(\pi_{0}^{*}B)_{X'}=\pi_{1}^{*}B_{Z_{1}}$.
Recall that $f:Z_{0}\to V_{0}$ is the extremal contraction of the $D$-MMP. 
Let $f_{1}:Z_{1}\to V_{0}$ be the induced morphism. 
By construction, there is a rational number $r$ and $\mathbb{Q}$-Cartier $\mathbb{Q}$-divisor $G$ on $V_{0}$ satisfying $B-rD\sim_{\mathbb{Q}}f^{*}G$.
Then $\pi_{0}^{*}B-r(K_{X_{0}}+\Delta_{X_{0}})\sim_{\mathbb{Q}}\pi_{0}^{*}f^{*}G.$
By taking the birational transform on $X'$, we obtain 
$(\pi_{0}^{*}B)_{X'}-r\pi_{1}^{*}D_{Z_{1}}\sim_{\mathbb{Q}}\pi_{1}^{*}f_{1}^{*}G$
because $K_{X'}+\Delta_{X'} \sim_{\mathbb{Q}} \pi_{1}^{*} D_{Z_{1}}$.
Since $B_{Z_{1}}-rD_{Z_{1}}\sim_{\mathbb{Q}}f_{1}^{*}G$, we see that  $(\pi_{0}^{*}B)_{X'}\sim_{\mathbb{Q}}\pi_{1}^{*}B_{Z_{1}}$. 
Next we prove $(\pi_{0}^{*}B)_{X'}=\pi_{1}^{*}B_{Z_{1}}$ as $\mathbb{Q}$-divisors. 
We note that $B$ or $-B$ is nef over $V_{0}$ because the relative Picard number $\rho(Z_{0}/V_{0})$ is one. 
Let $p:\widetilde{Z} \to Z_{0}$ and $p':\widetilde{Z} \to Z_{1}$ be a common resolution of $Z_{0}\dashrightarrow Z_{1}$, and let $q:\widetilde{X}\to X_{0}$ and $q':\widetilde{X}\to X'$ be a common resolution of $\phi:X_{0}\dashrightarrow X'$ such that the induced map $h:\widetilde{X} \to \widetilde{Z}$ is a morphism.
We set $F=p^{*}B-p'^{*}B_{Z_{1}}$. 
Then $F$ or $-F$ is effective by the negativity lemma. 
Moreover, by construction, we have 
\begin{equation*}
\begin{split}
(\pi_{0}^{*}B)_{X'}-\pi_{1}^{*}B_{Z_{1}}&=q'_{*}q^{*}\pi_{0}^{*}B-q'_{*}q'^{*}\pi_{1}^{*}B_{Z_{1}}\\
&=q'_{*}(h^{*}p^{*}B)-q'_{*}(h^{*}p'^{*}B_{Z_{1}})=q'_{*}h^{*}F.
\end{split}
\end{equation*}
On the other hand, since $(\pi_{0}^{*}B)_{X'}\sim_{\mathbb{Q}}\pi_{1}^{*}B_{Z_{1}}$, we have $q'_{*}h^{*}F\sim_{\mathbb{Q}}0$. 
Then $q'_{*}h^{*}F=0$ because $F$ or $-F$ is effective. 
In this way, we see that $(\pi_{0}^{*}B)_{X'}=\pi_{1}^{*}B_{Z_{1}}$ as $\mathbb{Q}$-divisors. 

Now we have $(\pi_{0}^{*}A)_{X'}=\pi_{1}^{*}A_{Z_{1}}$.
Since $\bigl(X',\Delta_{X'}+\lambda_{1}(\pi_{0}^{*}A)_{X'}\bigr)$ is lc, $(X',\Delta_{X'}+\lambda_{1}\pi_{1}^{*}A_{Z_{1}})$ is lc.
Therefore we can apply the above arguments to $\pi_{1}:(X_{k_{1}},\Delta_{X_{k_{1}}})=(X',\Delta_{X'})\to Z_{1}$ and $\lambda_{1}A_{Z_{1}}$. 
By repeating these arguments, we have the following diagram 
$$
\xymatrix@C=21pt{
(X_{0},\Delta_{X_{0}}) \ar_{\pi_{0}}[d] \ar@{-->}[r]&
\cdots \ar@{-->}[r]&(X_{k_{1}},\Delta_{X_{k_{1}}})\ar_{\pi_{1}}[d] \ar@{-->}[r]&
\cdots\ar@{-->}[r]& (X_{k_{i}},\Delta_{X_{k_{i}}}) \ar_{\pi_{i}}[d] \ar@{-->}[r]&
\cdots \\
(Z_{0}, \lambda_{0}) \ar@{-->}[rr]&&
(Z_{1}, \lambda_{1}) \ar@{-->}[r]&
\cdots\ar@{-->}[r]&(Z_{i}, \lambda_{i})\ar@{-->}[r]&
\cdots
}
$$
such that
\begin{itemize}
\item
for any $i$, $\pi_{i}$ is projective and surjective with connected fibers, 
\item 
the upper horizontal sequence of birational maps is a sequence of steps of the $(K_{X_{0}}+\Delta_{X_{0}})$-MMP with scaling of $\pi_{0}^{*}A$ such that if we set $k_{0}=0$ and  
$$\lambda'_{j}={\rm inf}\{\mu \in \mathbb{R}_{\geq0} \,|\, K_{X_{j}}+\Delta_{X_{j}}+\mu(\pi_{0}^{*}A)_{X_{j}}{\rm \;is \;nef}\},$$
then $\lambda'_{j}=\lambda_{i}$ for any $k_{i}\leq j<k_{i+1}$, and 
\item
$(\pi_{i}^{*}B)_{X_{k_{i+1}}}=\pi_{i+1}^{*}B_{Z_{i+1}}$ for any $i$ and any $\mathbb{Q}$-divisor $B$ on $Z_{i}$.
\end{itemize}
Pick any two indices $i<i'$ and $\mathbb{Q}$-divisor $B$ on $Z_{i}$. 
Then we can check that $(\pi_{i}^{*}B)_{X_{k_{i'}}}=\pi_{i'}^{*}B_{Z_{i'}}$ by induction on $i'-i$. 
Therefore the diagram satisfies all the conditions of the lemma. 
\end{proof}

\section{Proof of the main result in klt case and a reduction}\label{secreduction}

In this section we prove Theorem \ref{thmmain} in a special case, which contains the klt case, and prove a reduction lemma. 

Proposition \ref{propkltcase} below is a special case of Theorem \ref{thmmain}. 
From this proposition we see that Theorem \ref{thmmain} holds when $(X,\Delta)$ is klt.

\begin{prop}\label{propkltcase}
Fix a positive integer $d$, and assume the existence of a good minimal model or a Mori fiber space for all projective Kawamata log terminal pairs of dimension $d$ with boundary $\mathbb{Q}$-divisors.  

Let $\pi:X \to Z$ be a projective surjective morphism of normal projective varieties such that ${\rm dim}\,Z=d$. 
Let $(X,\Delta)$ be a log canonical pair such that $\Delta$ is a $\mathbb{Q}$-divisor and every lc center of $(X,\Delta)$ dominates $Z$.
Suppose that $K_{X}+\Delta \sim_{\mathbb{Q},\,Z}0$.

Then $(X,\Delta)$ has a good minimal model or a Mori fiber space.
\end{prop}

\begin{proof}
We can prove this by the same arguments as in the proof of \cite[Proposition 3.3]{birkarhu-arg}. 
We write details for the reader's convenience. 

By taking the Stein factorization of $\pi$, we may assume that $\pi$ has connected fibers. 
We may also assume that $K_{X}+\Delta$ is pseudo-effective because otherwise we can find a  Mori fiber space of $(X,\Delta)$ by \cite{bchm}. 
By \cite[Corollary 3.2]{fg-bundle}, there is a $\mathbb{Q}$-divisor $\Psi$ on $Z$ such that $(Z,\Psi)$ is klt and $K_{X}+\Delta \sim_{\mathbb{Q}}\pi^{*}(K_{Z}+\Psi)$. 
Then $K_{Z}+\Psi$ is pseudo-effective (cf.~\cite[II 5.6 Lemma]{nakayama-zariski-decom}).
By the hypothesis, there is a good minimal model $\phi:(Z,\Psi)\dashrightarrow (Z',\Psi_{Z'})$ of $(Z,\Psi)$. 

Let $f:W \to Z$ and $f':W \to Z'$ be a common resolution of $\phi$ and let $g:(Y,\Gamma) \to (X,\Delta)$ be a log smooth model such that the induced map $h:Y\dashrightarrow W$ is a morphism. 
Then we see that $f'\circ h:Y \to Z'$ has connected fibers.  
Moreover we have $K_{Y}+\Gamma = g^{*}(K_{X}+\Delta)+E$ for an effective $g$-exceptional divisor $E$ and $f^{*}(K_{Z}+\Psi)=f'^{*}(K_{Z'}+\Psi_{Z'})+F$ for an effective $f'$-exceptional divisor $F$. 
Then
\begin{equation*}
\begin{split}
K_{Y}+\Gamma &= g^{*}(K_{X}+\Delta)+E \sim_{\mathbb{Q}}g^{*}\pi^{*}(K_{Z}+\Psi)+E\\
&=h^{*}f^{*}(K_{Z}+\Psi)+E=h^{*}f'^{*}(K_{Z'}+\Psi_{Z'})+h^{*}F+E.
\end{split}
\end{equation*}
We run the $(K_{Y}+\Gamma)$-MMP over $Z'$ with scaling of an ample divisor 
$$Y \dashrightarrow Y_{1}\dashrightarrow \cdots\dashrightarrow Y_{i}\dashrightarrow \cdots.$$
Pick an open set $U$ of $Z$ such that the restriction of $\phi$ to $U$ is an isomorphism $\phi|_{U}:U\to \phi(U)$ and the codimension of $Z'\setminus \phi(U)$ in $Z'$ is at least two. 
By shrinking $U$ if necessary, we can assume that $F$ is mapped into $Z'\backslash \phi(U)$. 
Set $V=(\pi\circ g)^{-1}(U)$.
Since $K_{X}+\Delta \sim_{\mathbb{Q},\,Z}0$ and by the definition of log smooth models, we see that  $(\pi^{-1}(U),\Delta|_{\pi^{-1}(U)})$ is a weak lc model model of $(V,\Gamma|_{V})$ over $U$ with relatively trivial log canonical divisor. 
Since $U \simeq \phi(U)$, the $(K_{Y}+\Gamma)$-MMP over $Z'$ must terminate over $\phi(U)$. 
In other words, if $V_{i}$ denotes the inverse image of $\phi(U)$ on $Y_{i}$, the divisor $(K_{Y_{i}}+\Gamma_{Y_{i}})|_{V_{i}}$ is $\mathbb{Q}$-linearly equivalent to the pullback of $(K_{Z'}+\Psi_{Z'})|_{\phi(U)}$ for any $i\gg0$. 
Therefore $E$ is eventually contracted over $\phi(U)$. 

By the above facts, we see that $(h^{*}F)_{Y_{i}}+E_{Y_{i}}$ is mapped into $Z'\backslash \phi(U)$ for any $i\gg 0$. 
In particular $(h^{*}F)_{Y_{i}}+E_{Y_{i}}$ is a very exceptional divisor over $Z'$ (cf.~\cite[Definition 3.1]{birkar-flip}). 
Moreover, by the definition of the log MMP with scaling, $K_{Y_{i}}+\Gamma_{Y_{i}}\sim_{\mathbb{Q},\,Z'}(h^{*}F)_{Y_{i}}+E_{Y_{i}}$ is the limit of movable divisors over $Z'$ for any $i \gg 0$. 
Then $(h^{*}F)_{Y_{i}}+E_{Y_{i}}=0$ by \cite[Lemma 3.3]{birkar-flip}. 
Therefore $K_{Y_{i}}+\Gamma_{Y_{i}}$ is $\mathbb{Q}$-linearly equivalent to the pullback of $K_{Z'}+\Psi_{Z'}$ for some $i$. 
Since $K_{Z'}+\Psi_{Z'}$ is semi-ample, $K_{Y_{i}}+\Gamma_{Y_{i}}$ is also semi-ample. 
Therefore $(Y_{i},\Gamma_{Y_{i}})$ is a good minimal model of $(Y,\Gamma)$. 
Since $(Y,\Gamma)$ is a log smooth model of $(X,\Delta)$, $(X,\Delta)$ also has a good minimal model. 
So we are done. 
\end{proof} 

We can prove the following proposition by using \cite[Corollary 1.4.2]{bchm} and the same discussion as in the proof of Proposition \ref{propkltcase}. 

\begin{prop}\label{propbigkltcase}
Let $\pi:X \to Z$ be a projective surjective morphism of $\mathbb{Q}$-factorial normal projective varieties with connected fibers. 
Let $(X,\Delta)$ be a log canonical pair such that $\Delta$ is a $\mathbb{Q}$-divisor and every lc center of $(X,\Delta)$ dominates $Z$.
Suppose that $K_{X}+\Delta \sim_{\mathbb{Q}}\pi^{*}D$ for a $\mathbb{Q}$-divisor $D$ on $Z$.

If $D$ is big, then $(X,\Delta)$ has a good minimal model.
\end{prop}

We close this section with a reduction lemma, which plays a key role in the proof of Theorem \ref{thmmain}. 

\begin{lemm}\label{lemreduction}
To prove Theorem \ref{thmmain}, we may assume the following conditions about $\pi:(X,\Delta)\to Z$. 
\begin{enumerate}
\item[(i)]
$\pi$ has connected fibers, $(X,0)$ is $\mathbb{Q}$-factorial Kawamata log terminal and  $Z$ is $\mathbb{Q}$-factorial, 
\item[(ii)]
$(X,\Delta=\Delta'+\Delta'')$ satisfies all the conditions of Lemma \ref{lemlogresol} with respect to $\pi$, 
\item[(iii)]
$K_{X}+\Delta \sim_{\mathbb{Q}}\pi^{*}D$ and $\Delta'' \sim_{\mathbb{Q}}\pi^{*}E$, where $D$ is a pseudo-effective $\mathbb{Q}$-divisor and $E$ is an effective $\mathbb{Q}$-divisor on $Z$, and 
\item[(iv)]
there is a $\mathbb{Q}$-divisor $A$ on $X$ such that $K_{X}+\Delta+\delta A$ is movable for any sufficiently small $\delta>0$. 
\end{enumerate} 
\end{lemm}

\begin{proof}
By taking a dlt blow-up and by replacing $(X,\Delta)$ if necessary, we can assume that $X$ is $\mathbb{Q}$-factorial and $(X,0)$ is klt. 
We may also assume that $K_{X}+\Delta$ is pseudo-effective because otherwise we can find a Mori fiber space of $(X,\Delta)$ by running the $(K_{X}+\Delta)$-MMP with scaling. 
We note that existence of a good minimal model of $(X, \Delta)$ is equivalent to existence of a weak lc model of $(X, \Delta)$ with semi-ample log canonical divisor (see \cite[Corollary 3.7]{birkar-flip}). 
By taking the Stein factorization of $\pi$ and by Corollary \ref{cordltblowup}, we may assume that $\pi$ has connected fibers and the  condition (ii) of the lemma holds. 

Next we show that we can assume $\Delta''\sim_{\mathbb{Q},\,Z}0$ to prove Theorem \ref{thmmain}. 
Since $K_{X}+\Delta$ is pseudo-effective and $\Delta''$ is vertical over $Z$, we see that $K_{X}+\Delta'$ is pseudo-effective over $Z$. 
We run the $(K_{X}+\Delta')$-MMP over $Z$ with scaling of an ample divisor. 
By \cite[Theorem 1.1]{birkar-flip}, this log MMP terminates with a good minimal model $\phi:(X,\Delta' )\to (X',\Delta'_{X'})$ of $(X,\Delta')$ over $Z$ because $K_{X}+\Delta'+\Delta'' \sim_{\mathbb{Q},\,Z}0$. 
Set $\Delta_{X'}''=\phi_{*}\Delta''$ and  $\Delta_{X'}=\phi_{*}\Delta$.  
Then $(X',\Delta_{X'}=\Delta_{X'}'+\Delta_{X'}'')$ satisfies all the conditions of Lemma \ref{lemlogresol} with respect to the morphism $X' \to Z$. 
Let $\pi':X' \to Z'$ be the Stein factorization of the morphism induced by $K_{X'}+\Delta_{X'}'$ over $Z$. 
Then $\pi'$ has connected fibers and it is easy to see that the morphism $Z'\to Z$ is birational. 
Therefore $(X',\Delta_{X'}=\Delta_{X'}'+\Delta_{X'}'')$ satisfies all the conditions of Lemma \ref{lemlogresol} with respect to $\pi'$. 
Moreover we have $K_{X'}+\Delta_{X'}\sim_{\mathbb{Q},\,Z'}0$  and $K_{X'}+\Delta'_{X'}\sim_{\mathbb{Q},\,Z'}0$. 
Therefore $\Delta''_{X'} \sim_{\mathbb{Q},\,Z'}0$. 
Since $\phi$ is a birational contraction and both  $K_{X}+\Delta$ and $K_{X'}+\Delta_{X'}$ are $\mathbb{Q}$-linearly equivalent to the pullback of the same divisor on $Z$, $(X,\Delta)$ has a weak lc model with semi-ample log canonical divisor if $(X',\Delta_{X'})$ has a weak lc model with semi-ample log canonical divisor.
In this way, by replacing $\pi:(X,\Delta)\to Z$ with $\pi':(X',\Delta_{X'})\to Z'$, we may assume that $\Delta''\sim_{\mathbb{Q},\,Z}0$. 
Then there is an effective $\mathbb{Q}$-Cartier $\mathbb{Q}$-divisor $E$ on $Z$ such that $\Delta'' \sim_{\mathbb{Q}}\pi^{*}E$.

Now we can prove that we may assume the condition (i) of the lemma to prove Theorem \ref{thmmain}. 
Let $E$ be a $\mathbb{Q}$-divisor defined above and pick a $\mathbb{Q}$-divisor $D$ on $Z$ such that $K_{X}+\Delta \sim_{\mathbb{Q}}\pi^{*}D$.  
By the condition (ii) of this lemma and the condition (iii) of Lemma \ref{lemlogresol}, every lc center of $(X,\Delta')$ dominates $Z$. 
Since $K_{X}+\Delta'\sim_{\mathbb{Q},\,Z}0$, by \cite[Corollary 3.2]{fg-bundle}, there exists a klt pair on $Z$. 
Let $f:\widetilde{Z}\to Z$ be a dlt blow-up of the klt pair. 
By construction $\widetilde{Z}$ is $\mathbb{Q}$-factorial and $Z$ and $\widetilde{Z}$ are isomorphic in codimension one. 
Let $g:(Y,\Gamma') \to (X,\Delta')$ be a log smooth model of $(X,\Delta')$ such that the induced map $h: Y\dashrightarrow \widetilde{Z}$ is a morphism. 
Then $\pi \circ g =f \circ h$ and we can write 
$$K_{Y}+\Gamma'=g^{*}(K_{X}+\Delta')+F\sim_{\mathbb{Q}}(f \circ h)^{*}(D-E)+F$$
with an effective $g$-exceptional divisor $F$ which contains every $g$-exceptional prime divisor whose discrepancy with respect to $(X,\Delta')$ is greater than $-1$.  
We run the $(K_{Y}+\Gamma')$-MMP over $\widetilde{Z}$ with scaling. 
Since $\widetilde{Z}$ and $Z$ are isomorphic in codimension one, by the same argument as in the proof of Proposition \ref{propkltcase}, we obtain a good minimal model $\phi:(Y,\Gamma') \dashrightarrow (Y',\Gamma'_{Y'})$ over $\widetilde{Z}$. 
Let $h':Y' \to \widetilde{Z}$ be the induced morphism. 
Then $Y'$ is $\mathbb{Q}$-factorial, $(Y',0)$ is klt and $h'$ has connected fibers. 
We also have $F_{Y'}=0$ and $K_{Y'}+\Gamma'_{Y'}\sim_{\mathbb{Q}}(f\circ h')^{*}(D-E)$ by construction (see the proof of Proposition \ref{propkltcase}). 
Set 
$$\Gamma''_{Y'}=\phi_{*}g^{*}\Delta'' \sim_{\mathbb{Q}}(f \circ h')^{*}E \qquad {\rm and} \qquad \Gamma_{Y'}=\Gamma'_{Y'}+\Gamma''_{Y'}.$$ 
By taking a common resolution of $X\dashrightarrow Y'$, we see that $(Y',\Gamma_{Y'})$ is lc. 
Moreover we see that $\Gamma''_{Y'}$ is a reduced divisor. 
Indeed, we can write $g^{*}\Delta''=g_{*}^{-1}\Delta''+\sum a_{i}E_{i}$ where $E_{i}$ are $g$-exceptional prime divisors and $a_{i}>0$. 
Then $a(E_{i},X,\Delta')>-1$ because $(X,\Delta=\Delta'+\Delta'')$ is lc. 
Therefore ${\rm Supp}\,(\sum a_{i}E_{i})\subset {\rm Supp}\,F$. 
Since $F_{Y'}=0$, we see that $\Gamma''_{Y'}=\phi_{*}g_{*}^{-1}\Delta''$ and thus $\Gamma''_{Y'}$ is a reduced divisor. 
Now we can easily check that $(Y',\Gamma_{Y'}=\Gamma'_{Y'}+\Gamma''_{Y'})$ satisfies all the conditions of Lemma \ref{lemlogresol} with respect to $h'$. 
If we set $D'=f^{*}D$ and $E'=f^{*}E$, then $K_{Y'}+\Gamma_{Y'}\sim_{\mathbb{Q}}h'^{*}D'$ and $\Gamma''_{Y'}\sim_{\mathbb{Q}}h'^{*}E'$. 
By construction we can also check that $(Y',\Gamma_{Y'})$ is a log birational model of $(X,\Delta)$. 
Therefore, to prove Theorem \ref{thmmain} for $(X,\Delta)$, we only have to show that $(Y',\Gamma_{Y'})$ has a weak lc model with semi-ample log canonical divisor. 
In this way, by replacing $\pi:(X,\Delta) \to Z$ with $h':(Y',\Gamma_{Y'})\to \widetilde{Z}$, we may assume the condition (i) of the lemma. 

Finally we show that we can assume the condition (iv) of the lemma to prove Theorem \ref{thmmain}. 
Since $K_{X}+\Delta$ is pseudo-effective, $D$ is pseudo-effective (cf.~\cite[II 5.6 Lemma]{nakayama-zariski-decom}).
Let $A$ be a general ample $\mathbb{Q}$-divisor on $Z$ such that $A+E$ and $D+A$ are also ample. 
By Lemma \ref{lemd-mmp}, there is a sequence of birational  maps of the $D$-MMP with scaling of $A$ 
$$(Z=Z_{0},\lambda_{0}) \dashrightarrow \cdots \dashrightarrow (Z_{i},\lambda_{i})\dashrightarrow \cdots.$$
Then we have ${\rm lim}_{i \to \infty} \lambda_{i}=0$. 
Moreover, by construction of the $D$-MMP with scaling in Lemma \ref{lemd-mmp},  $D_{Z_{i}}+\lambda_{i}A_{Z_{i}}$ is semi-ample if $\lambda_{i}>0$. 

Since $A$ is a general ample $\mathbb{Q}$-divisor and $D+A$ is nef, by Lemma \ref{lemliftmmp}, we have the following diagram 
$$
\xymatrix@C=21pt{
(X,\Delta) \ar_{\pi=\pi_{0}}[d] \ar@{-->}[r]&
\cdots \ar@{-->}[r]&(X_{k_{1}},\Delta_{X_{k_{1}}})\ar_{\pi_{1}}[d] \ar@{-->}[r]&
\cdots\ar@{-->}[r]& (X_{k_{i}},\Delta_{X_{k_{i}}}) \ar_{\pi_{i}}[d] \ar@{-->}[r]&
\cdots \\
(Z_{0}, \lambda_{0}) \ar@{-->}[rr]&&
(Z_{1}, \lambda_{1}) \ar@{-->}[r]&
\cdots\ar@{-->}[r]&(Z_{i}, \lambda_{i})\ar@{-->}[r]&
\cdots
}
$$
such that 
\begin{enumerate}
\item[(i)]
for any $i$, $\pi_{i}$ is projective and surjective with connected fibers, 
\item[(ii)]
the upper horizontal sequence of birational maps is a sequence of steps of the $(K_{X}+\Delta)$-MMP with scaling of $\pi^{*}A$ such that 
if we set $k_{0}=0$ and
$$\lambda'_{j}={\rm inf}\{\mu \in \mathbb{R}_{\geq0} \,|\, K_{X_{j}}+\Delta_{X_{j}}+\mu (\pi^{*}A)_{X_{j}}{\rm \;is \;nef}\},$$
then $\lambda'_{j}=\lambda_{i}$ for $k_{i}\leq j<k_{i+1}$, and 
\item[(iii)]
for any two indices $i<i'$ and any ${\mathbb{Q}}$-divisor $B$ on $Z_{i}$, we have $(\pi_{i}^{*}B)_{X_{k_{i'}}}=\pi_{i'}^{*}B_{Z_{i'}}$.
\end{enumerate}
Set $A'=\pi^{*}A$. 
Since $(X,0)$ is klt and $X$ is $\mathbb{Q}$-factorial, $(X_{k_{i}},0)$ is klt and $X_{k_{i}}$ is $\mathbb{Q}$-factorial. 
We also have $K_{X_{k_{i}}}+\Delta_{X_{k_{i}}}\sim_{\mathbb{Q}}\pi_{i}^{*}D_{Z_{i}}$ and $\Delta_{X_{k_{i}}}''\sim_{\mathbb{Q}}\pi_{i}^{*}E_{Z_{i}}$ by (iii) of the above properties. 
Taking a common resolution of $X\dashrightarrow X_{k_{i}}$, we see that $(X_{k_{i}}, \Delta_{X_{k_{i}}}=\Delta
'_{X_{k_{i}}}+\Delta''_{X_{k_{i}}})$ satisfies all the conditions of Lemma \ref{lemlogresol} with respect to $\pi_{i}$.
In this way we can replace $\pi:(X,\Delta) \to Z$ with $\pi_{i}:(X_{k_{i}},\Delta_{X_{k_{i}}})\to Z_{i}$ for some $i \gg 0$ and hence we may assume that $K_{X}+\Delta$ is nef (i.e., the $(K_{X}+\Delta)$-MMP terminates after finitely many steps) or the $(K_{X}+\Delta)$-MMP contains only flips. 
We may also assume that $K_{X_{k_{i}}}+\Delta_{X_{k_{i}}}+\lambda_{i}A'_{X_{k_{i}}}\sim_{\mathbb{Q}}\pi_{i}^{*}(D_{Z_{i}}+\lambda_{i}A_{Z_{i}})$ is semi-ample if $\lambda_{i}>0$. 

If $K_{X}+\Delta$ is nef, then any ample $\mathbb{Q}$-divisor satisfies the condition (iv) of the lemma. 
If $\lambda_{i}>0$ for any $i$, the divisor $A'$ satisfies the condition of the lemma because 
$K_{X}+\Delta+\lambda_{i}A'$ is movable and ${\rm lim}_{i \to \infty} \lambda_{i}=0$. 
In this way we can assume the condition (iv) of the lemma. 
So we are done. 
\end{proof}

\section{Proof of the main result}\label{secproofmainthm}

In this section we prove Theorem \ref{thmmain}. 
First we prove Theorem \ref{thmmain} assuming Proposition \ref{propcase1}, Proposition \ref{propcase2} and Proposition \ref{propcase3}. 
After that, we prove Proposition \ref{propcase1}, Proposition \ref{propcase2} and Proposition \ref{propcase3}.

\begin{proof}[Proof of Theorem \ref{thmmain}]
We prove it by induction on $d_{0}$. 
Clearly Theorem \ref{thmmain} holds when $d_{0}=0$. 
Fix an integer $d_{0}>0$ and assume Theorem \ref{thmmain} for $d_{0}-1$, and pick any $\pi:(X,\Delta)\to Z$ as in Theorem \ref{thmmain}. 
By Lemma \ref{lemreduction}, we can assume that $\pi:(X,\Delta) \to Z$ satisfies all the conditions of Lemma \ref{lemreduction}. 
Moreover we may assume that $E\neq0$ because otherwise Theorem \ref{thmmain} follows from Proposition  \ref{propkltcase} and the condition (iii) of Lemma \ref{lemlogresol}. 
Then we have the following three cases.

\begin{case}\label{case1}
$D$ is not big and $D-e E$ is pseudo-effective for a sufficiently small positive rational number $e$. 
\end{case}

\begin{case}\label{case2}
$D$ is not big and $D-e E$ is not pseudo-effective for any positive rational number $e$.
\end{case}

\begin{case}\label{case3}
$D$ is big.
\end{case}
But then Theorem \ref{thmmain} follows from Proposition \ref{propcase1}, Proposition \ref{propcase2} and Proposition \ref{propcase3} below. 
So we are done. 
\end{proof}

Next we prove Proposition \ref{propcase1}, Proposition \ref{propcase2} and Proposition \ref{propcase3}.
From now on we freely use the notations of the conditions (ii) and (iii) of Lemma \ref{lemreduction}.

\begin{prop}\label{propcase1}
Fix a positive integer $d_{0}$ and assume Theorem \ref{thmmain} for $d_{0}-1$. 
Let $\pi:(X,\Delta) \to Z$ be as in Theorem \ref{thmmain} satisfying all conditions of Lemma \ref{lemreduction}.
Let $D$ and $E$ be as in the condition (iii) of Lemma \ref{lemreduction}. 

If $D$ is not big and $D-e E$ is pseudo-effective for a sufficiently small positive rational number $e$, then Theorem \ref{thmmain} holds for $(X,\Delta)\to Z$. 
\end{prop}

\begin{proof}
We can assume $E\neq0$. 
We prove it with several steps.

\begin{step2}\label{step1incase1}
In this step we construct a diagram used in the proof. 
 
Fix a strictly decreasing infinite sequence of rational numbers $\{a_{n}\}_{n\geq1}$ such that $0<a_{n}<e$ for any $n$ and ${\rm lim}_{n \to \infty}a_{n}=0$. 
Then $D-a_{n}E$ is pseudo-effective for all $n\geq 1$. 
By \cite[Corollary 3.2]{fg-bundle} and the definition of $D$ and $E$, there are $\mathbb{Q}$-divisors $\Psi_{n}$ on $Z$ such that all $(Z,\Psi_{n})$ are klt and $D-a_{n}E \sim_{\mathbb{Q}}K_{Z}+\Psi_{n}$.
Then $K_{X}+\Delta-a_{n}\Delta'' \sim_{\mathbb{Q}} \pi^{*}(K_{Z}+\Psi_{n})$.

By the hypothesis of Theorem \ref{thmmain}, $(Z,\Psi_{n})$ has a  good minimal model. 
By  \cite[Theorem 4.1 (iii)]{birkar-flip} and running the $(K_{Z}+\Psi_{n})$-MMP with scaling, we get a good minimal model $(Z,\Psi_{n})\dashrightarrow \bigl(Z_{n},(\Psi_{n})_{Z_{n}}\bigr)$ of $(Z,\Psi_{n})$.
By Lemma \ref{lemliftmmp}, we obtain the following diagram 
$$
\xymatrix{
(X,\Delta-a_{n}\Delta'') \ar[d]_{\pi}\ar@{-->}[r]^{\!\!\!\!\!\!\!\!\phi_{n}}&(X_{n},\Delta_{X_{n}}-a_{n}\Delta''_{X_{n}})\ar[d]_{\pi_{n}}\\
Z\ar@{-->}[r]&Z_{n}
}
$$
such that $\phi_{n}$ is a sequence of finitely many steps of the $(K_{X}+\Delta-a_{n}\Delta'')$-MMP with scaling. 
Then $K_{X_{n}}+\Delta_{X_{n}}-a_{n}\Delta_{X_{n}}''\sim_{\mathbb{Q}} \pi_{n}^{*}\bigl(K_{Z_{n}}+(\Psi_{n})_{Z_{n}}\bigr)$ 
by the condition (iii) of Lemma \ref{lemliftmmp}. 
Therefore $K_{X_{n}}+\Delta_{X_{n}}-a_{n}\Delta_{X_{n}}''$ is semi-ample and thus $(X_{n},\Delta_{X_{n}}-a_{n}\Delta''_{X_{n}})$ is a good minimal model of $(X,\Delta-a_{n}\Delta'')$.
We note that $X_{n}$ is $\mathbb{Q}$-factorial and $(X_{n},0)$ is klt, and $\pi_{n}:X_{n} \to Z_{n}$ is projective and surjective with connected fibers by the condition (i) of Lemma \ref{lemliftmmp}. 
We also note that $K_{X_{n}}+\Delta_{X_{n}}-t\Delta_{X_{n}}'' \sim_{\mathbb{Q}} \pi_{n}^{*}(D_{Z_{n}}-t E_{Z_{n}})$ for any $t\geq 0$ by the condition (iii) of Lemma \ref{lemliftmmp}. 
\end{step2}

\begin{step2}\label{step2incase1}
In this step we prove that there are infinitely many indices $n$ such that $(X_{n},\Delta_{X_{n}})$ is lc.

Suppose by contradiction that there are only finitely many indices $n$ such that $(X_{n},\Delta_{X_{n}})$ is lc. 
Fix $n_{0}$ such that $(X_{i},\Delta_{X_{i}})$ is not lc for every $i \geq n_{0}$. 
Consider 
$$I=\{M \in \mathbb{R}_{\geq0}\,|\, M ={\rm lct}(X_{i},\Delta'_{X_{i}};\Delta_{X_{i}}''),\; i\geq n_{0}\}$$
where $\Delta'_{X_{i}}$ is the birational transform of $\Delta'$ on $X_{i}$. 
Then $I$ does not contain one by our assumption. 
On the other hand, since $\Delta'_{X_{i}}$ is the birational transform of $\Delta'$ on $X_{i}$, any coefficient of component in $\Delta'_{X_{i}}$ is in a finite set which does not depend on $i$.  
Moreover $\Delta_{X_{i}}''$ is a reduced divisor and ${\rm lct}(X_{i},\Delta'_{X_{i}};\Delta_{X_{i}}'')\geq 1-a_{i}$ by construction. 
Since ${\rm lim}_{n \to \infty}a_{n}=0$, by the ACC for log canonical thresholds (cf.~Theorem \ref{thmacclct}), the set $I$ must contain one. 
In this way we get a contradiction and thus there are infinitely many indices $n$ such that $(X_{n},\Delta_{X_{n}})$ is lc.
\end{step2}

\begin{step2}\label{step3incase1}
By taking a common resolution of $\phi_{n}$, we can check that $(X_{n}, \Delta_{X_{n}}=\Delta
'_{X_{n}}+\Delta''_{X_{n}})$ satisfies all the conditions of Lemma \ref{lemlogresol} with respect to $\pi_{n}$.
In this way, by Step \ref{step1incase1} and Step \ref{step2incase1} we can get a strictly decreasing infinite sequence $\{a_{n}\}_{n\geq1}$ of positive rational numbers such that 
\begin{enumerate}
\item[(i)]
$a_{n}<e$ for any $n$ and ${\rm lim}_{n\to \infty}a_{n}=0$, and 
\item[(ii)]
for any $n\geq1$, there is a diagram
$$
\xymatrix{
(X,\Delta-a_{n}\Delta'') \ar[d]_{\pi}\ar@{-->}[r]^{\!\!\!\!\!\!\!\!\!\!\phi_{n}}&(X_{n},\Delta_{X_{n}}-a_{n}\Delta''_{X_{n}})\ar[d]_{\pi_{n}}\\
Z\ar@{-->}[r]&Z_{n}
}
$$
such that 
\begin{enumerate}
\item[(ii-a)]
$X_{n}$ and $Z_{n}$ are $\mathbb{Q}$-factorial, $(X_{n},\Delta_{X_{n}})$ is lc, $(X_{n},0)$ is klt and $\pi_{n}$ is a projective surjective morphism with connected fibers,
\item[(ii-b)]
$(X_{n}, \Delta_{X_{n}}=\Delta'_{X_{n}}+\Delta''_{X_{n}})$ satisfies all the conditions of Lemma \ref{lemlogresol} with respect to $\pi_{n}$, $K_{X_{n}}+\Delta_{X_{n}} \sim_{\mathbb{Q}}\pi_{n}^{*}D_{Z_{n}}$ and $\Delta''_{X_{n}}\sim_{\mathbb{Q}}\pi_{n}^{*}E_{Z_{n}}$, and 
\item[(ii-c)]
$\phi_{n}$ is a sequence of finitely many steps of the $(K_{X}+\Delta-a_{n}\Delta'')$-MMP to a good minimal model $(X_{n},\Delta_{X_{n}}-a_{n}\Delta''_{X_{n}})$.
\end{enumerate} 
\end{enumerate} 
By replacing $\{a_{n}\}_{n\geq1}$ with its subsequence again, we also have that  
\begin{enumerate}
\item[(ii-d)]
$X_{i}$ and $X_{j}$ are isomorphic in codimension one for any $i$ and $j$.
\end{enumerate}
Indeed, for any $n\geq1$, every prime divisor contracted by $\phi_{n}$ is a component of $N_{\sigma}(K_{X}+\Delta-a_{n}\Delta'')$. 
Moreover $K_{X}+\Delta-e\Delta''$ is pseudo-effective by the choice of $e$. 
Therefore, by the basic property of $N_{\sigma}(\,\cdot\,)$, we have 
\begin{equation*}
\begin{split}
&N_{\sigma}(K_{X}+\Delta-a_{n}\Delta'')\\
&\leq \Bigl(1-\frac{a_{n}}{e} \Bigr) N_{\sigma}(K_{X}+\Delta)+\frac{a_{n}}{e}N_{\sigma}(K_{X}+\Delta-e\Delta'').
\end{split}
\end{equation*}
Thus every prime divisor contracted by $\phi_{n}$ is also a component of $N_{\sigma}(K_{X}+\Delta)+N_{\sigma}(K_{X}+\Delta-e\Delta'')$, which does not depend on $n$. 
Therefore we can replace $\{a_{n}\}_{n\geq1}$ with its subsequence so that $X_{i}$ and $X_{j}$ are isomorphic in codimension one for any $i$ and $j$.

We note that $D_{Z_{n}}-(a_{n}-\delta)E_{Z_{n}}$ is not big for any sufficiently small rational number $\delta >0$. 
Indeed, the birational map $Z \dashrightarrow Z_{n}$ is a sequence of finitely many steps of the $(D-a_{n}E)$-MMP. 
Then $Z \dashrightarrow Z_{n}$ is also  a sequence of finitely many steps of the $\bigl(D-(a_{n}-\delta)E\bigr)$-MMP for any sufficiently small $\delta>0$. 
Since $D-tE$ is not big for any $t\geq0$, we see that $D_{Z_{n}}-(a_{n}-\delta)E_{Z_{n}}$ is not big for any sufficiently small $\delta>0$. 
\end{step2}

\begin{step2}\label{step4incase1}
Suppose that $(X_{1},\Delta_{X_{1}})$ has a good minimal model. 
Then we can show that $(X,\Delta)$ has a weak lc model with semi-ample log canonical divisor. 
Indeed, by \cite[Theorem 4.1 (iii)]{birkar-flip}, any $(K_{X_{1}}+\Delta_{X_{1}})$-MMP with scaling of an ample divisor terminates. 
So we can get a sequence of finitely many steps of the $(K_{X_{1}}+\Delta_{X_{1}})$-MMP to a good minimal model $\psi:(X_{1},\Delta_{X_{1}})\dashrightarrow (X',\Delta_{X'})$.  
By construction $K_{X'}+\Delta_{X'}$ is semi-ample. 
Fix a sufficiently small positive rational number $t\ll a_{1}$.
Then $\psi$ is also a sequence of finitely many steps of the $(K_{X_{1}}+\Delta_{X_{1}}-t\Delta_{X_{1}}'')$-MMP.
We note that any lc center of $(X_{1},\Delta_{X_{1}}-t\Delta_{X_{1}}'')$ dominates $Z_{1}$.
By Proposition \ref{propkltcase}, $(X_{1},\Delta_{X_{1}}-t\Delta_{X_{1}}'')$ has a good minimal model, and thus $(X', \Delta_{X'}-t\Delta_{X'}'')$ has a good minimal model.
Therefore we can run the $(K_{X'}+\Delta_{X'}-t\Delta''_{X'})$-MMP with scaling of an ample divisor and obtain a good minimal model 
$\psi':(X', \Delta_{X'}-t\Delta_{X'}'')\dashrightarrow (X'', \Delta_{X''}-t\Delta_{X''}'').$
Now we get the following sequence of birational maps 
$$X \overset{\phi_{1}}{\dashrightarrow} X_{1}\overset{\psi}{\dashrightarrow} X'\overset{\psi'}{\dashrightarrow} X'',$$
where $\phi_{1}$ (resp.~$\psi$, $\psi'$) is a sequence of steps the $(K_{X}+\Delta-a_{1}\Delta'')$-MMP (resp.~the $(K_{X_{1}}+\Delta_{X_{1}})$-MMP, the $(K_{X'}+\Delta_{X'}-t\Delta''_{X'})$-MMP) to a good minimal model.
Since we pick $t>0$ sufficiently small, by the standard argument of the length of extremal rays, we see that $K_{X''}+\Delta_{X''}$ is also semi-ample (see, for example, the proof of \cite[Proposition 3.2 (5)]{birkar-existII} or the proof of \cite[Theorem 1.2]{has-ab}). 

We prove that $X_{1}$ and $X''$ are isomorphic in codimension one. 
More precisely, we prove that both $\psi$ and $\psi'$ contain only flips. 
We note that $\phi_{1}$ is in particular a birational contraction. 
Recall that there is a $\mathbb{Q}$-divisor $A$ on $X$ such that $K_{X}+\Delta+\delta A$ is movable for any sufficiently small $\delta>0$, which is the condition (iv) of Lemma \ref{lemreduction}.  
Therefore we see that $K_{X_{1}}+\Delta_{X_{1}}+\delta A_{X_{1}}$ is movable for any sufficiently small $\delta>0$. 
Then $N_{\sigma}(K_{X_{1}}+\Delta_{X_{1}})=0$, and $\psi$ contains only flips.
Moreover $K_{X'}+\Delta_{X'}-a_{1}\Delta_{X'}''$ is movable because of the condition (ii-c) of Step \ref{step3incase1}. 
So $N_{\sigma}(K_{X'}+\Delta_{X'}-a_{1}\Delta_{X'}'')=0$.
We also have $N_{\sigma}(K_{X'}+\Delta_{X'})=0$ since $K_{X'}+\Delta_{X'}$ is semi-ample. 
Since $t$ satisfies $0<t<a_{1}$, we have
\begin{equation*}
\begin{split}
&N_{\sigma}(K_{X'}+\Delta_{X'}-t\Delta''_{X'})\\
&\leq \Bigl(1-\frac{t}{a_{1}}\Bigr) N_{\sigma}(K_{X'}+\Delta_{X'})+\frac{t}{a_{1}}N_{\sigma}(K_{X'}+\Delta_{X'}-a_{1}\Delta_{X'}'')\\
&=0
\end{split}
\end{equation*}
and $\psi'$ contains only flips. 
Thus we see that  $X_{1}$ and $X''$ are isomorphic in codimension one. 

Recall that $X_{i}$ and $X_{j}$ are isomorphic in codimension one, which is the condition (ii-d) of Step \ref{step3incase1}. 
Therefore $X_{n}$ and $X''$ are isomorphic in codimension one for any $n$. 
By the condition (i) in Step \ref{step3incase1}, we have $t\geq a_{n}$ for any $n \gg 0$, where $\{a_{n}\}_{n\geq1}$ is defined in Step \ref{step3incase1}. 
Since $K_{X''}+\Delta_{X''}$ and $K_{X''}+\Delta_{X''}-t\Delta''_{X''}$ are semi-ample, $K_{X''}+\Delta_{X''}-a_{n}\Delta''_{X''}$ is also semi-ample for any $n \gg 0$. 
Now we recall that $(X_{n},\Delta_{X_{n}}-a_{n}\Delta''_{X_{n}})$ is a good minimal model of $(X,\Delta-a_{n}\Delta'')$.
From these facts, we see that $(X'',\Delta_{X''}-a_{n}\Delta''_{X''})$ is also a good minimal model of $(X,\Delta-a_{n}\Delta'')$ for any $n\gg0$. 
Let $p:Y \to X$ and $q:Y \to X''$ be a common resolution of $X\dashrightarrow X''$. 
Then 
$$p^{*}(K_{X}+\Delta-a_{n}\Delta'')-q^{*}(K_{X''}+\Delta_{X''}-a_{n}\Delta_{X''}'')\geq0$$
for any $n\gg0$.
Since ${\rm lim}_{n \to \infty}a_{n}=0$, we have 
$$p^{*}(K_{X}+\Delta)-q^{*}(K_{X''}+\Delta_{X''})\geq0$$ 
by considering the limit. 
Since $K_{X''}+\Delta_{X''}$ is semi-ample, we see that $(X'',\Delta_{X''})$ is a weak lc model of $(X,\Delta)$ with semi-ample log canonical divisor. 

In this way, to prove Proposition \ref{propcase1}, we only have to prove that $(X_{1},\Delta_{X_{1}})$ has a good minimal model.
\end{step2}

\begin{step2}\label{step5incase1}
Finally we prove that $(X_{1},\Delta_{X_{1}})$ has a good minimal model. 
If $E_{Z_{1}}= 0$, then $K_{X_{1}}+\Delta_{X_{1}}=K_{X_{1}}+\Delta_{X_{1}}-a_{1}\Delta''_{X_{1}}$ is semi-ample. 
Therefore we may assume that $E_{Z_{1}}\neq 0$.  
Recall again that any lc center of $(X_{1},\Delta_{X_{1}}-t\Delta_{X_{1}}'')$ dominates $Z_{1}$ for any $0<t \leq a_{1}$.

Pick a sufficiently small positive rational number $u \ll a_{1}$. 
By \cite[Corollary 3.2]{fg-bundle}, there is a $\mathbb{Q}$-divisor $\Psi'$ on $Z_{1}$ such that $(Z_{1},\Psi')$ is klt and
$D_{Z_{1}}-(a_{1}-u)E_{Z_{1}} \sim_{\mathbb{Q}}K_{Z_{1}}+\Psi'$.
Then $(Z_{1},\Psi')$ has a  good minimal model by the hypothesis. 
Run the $(K_{Z_{1}}+\Psi')$-MMP with scaling of an ample divisor. 
By \cite[Theorem 4.1 (iii)]{birkar-flip}, we obtain a good minimal model $(Z_{1},\Psi')\dashrightarrow(Z'_{1},\Psi'_{Z'_{1}})$ of $(Z_{1},\Psi')$.
Furthermore, by Lemma \ref{lemliftmmp}, we obtain the following diagram 
$$
\xymatrix{
\bigl(X_{1},\Delta_{X_{1}}-(a_{1}-u)\Delta''_{X_{1}}\bigr) \ar[d]_{\pi_{1}}\ar@{-->}[r]^{\phi}&\bigl(X'_{1},\Delta_{X'_{1}}-(a_{1}-u)\Delta''_{X'_{1}}\bigr)\ar[d]_{\pi'_{1}}\\
Z_{1}\ar@{-->}[r]&Z'_{1}
}
$$
such that the upper horizontal birational map $\phi$ is a sequence of finitely many steps of the $\bigl(K_{X_{1}}+\Delta_{X_{1}}-(a_{1}-u)\Delta''_{X_{1}} \bigr)$-MMP to a good minimal model. 
Since $K_{X_{1}}+\Delta_{X_{1}}-a_{1}\Delta''_{X_{1}}$ is semi-ample and $u$ is sufficiently small, by the standard argument of the length of extremal rays, we see that $\phi$ is also a sequence of finitely many steps of the $(K_{X_{1}}+\Delta_{X_{1}})$-MMP and $K_{X'_{1}}+\Delta_{X'_{1}}-a_{1}\Delta''_{X'_{1}}$ is semi-ample (cf.~\cite[Proposition 3.2 (5)]{birkar-existII}). 
From this we also see that $(X'_{1},\Delta_{X'_{1}})$ is lc.
Now we have 
\begin{equation*}
\begin{split}
K_{X'_{1}}+\Delta_{X'_{1}}-(a_{1}-u)\Delta''_{X'_{1}}&\sim_{\mathbb{Q}}\pi_{1}'^{*}\bigl(D_{Z'_{1}}-(a_{1}-u)E_{Z'_{1}}\bigr) {\rm \quad and} \\ K_{X'_{1}}+\Delta_{X'_{1}}-a_{1}\Delta_{X'_{1}}''&\sim_{\mathbb{Q}}\pi_{1}'^{*}(D_{Z'_{1}}-a_{1}E_{Z'_{1}}),
\end{split}
\end{equation*}
where $D_{Z'_{1}}-(a_{1}-u)E_{Z'_{1}}$ and $D_{Z'_{1}}-a_{1}E_{Z'_{1}}$ are semi-ample.

Recall that $D_{Z_{1}}-(a_{1}-\delta)E_{Z_{1}}$ is not big for any sufficiently small rational number $\delta>0$. 
Therefore $D_{Z'_{1}}-(a_{1}-\delta)E_{Z'_{1}}$ is not big for any sufficiently small rational number $\delta>0$. 
Pick two sufficiently small rational numbers $u_{1}$ and $u_{2}$ satisfying $0<u_{1}<u_{2}<u$. 
Then we see that $D_{Z'_{1}}-(a_{1}-u_{i})E_{Z'_{1}}$ is semi-ample for $i=1,2$ because these are represented by a $\mathbb{Q}_{>0}$-linear combination of $D_{Z'_{1}}-(a_{1}-u)E_{Z'_{1}}$ and $D_{Z'_{1}}-a_{1}E_{Z'_{1}}$. 
Moreover $D_{Z'_{1}}-(a_{1}-u_{i})E_{Z'_{1}}$ is not big. 
For $i=1,2$, let $f_{i}:Z'_{1} \to W_{i}$ be the Stein factorization of the projective morphism induced by $D_{Z'_{1}}-(a_{1}-u_{i})E_{Z'_{1}}$.
Then $W_{1}\simeq W_{2}$. 
Indeed, let $C$ be a curve on $Z'_{1}$. 
Then 
\begin{equation*}
\begin{split}
&\,C{\rm \; is \;contracted\;by\;}f_{1} \\
\Leftrightarrow &\,C\cdot \bigl(D_{Z'_{1}}-(a_{1}-u_{1})E_{Z'_{1}}\bigr)=0\\
\Leftrightarrow &\,C\cdot (D_{Z'_{1}}-a_{1}E_{Z'_{1}})=C\cdot \bigl(D_{Z'_{1}}-(a_{1}-u)E_{Z'_{1}}\bigr)=0\\
\Leftrightarrow &\,C\cdot \bigl(D_{Z'_{1}}-(a_{1}-u_{2})E_{Z'_{1}} \bigr)=0\\
\Leftrightarrow &\,C {\rm \; is \;contracted\;by\;}f_{2}.
\end{split}
\end{equation*}
Thus $W_{1} \simeq W_{2}$. 
Set $f:Z'_{1} \to W=W_{1}\simeq W_{2}$.
By construction we have ${\rm dim}\, W<{\rm dim}\,Z'_{1}$ and $D_{Z'_{1}}-(a_{1}-u_{i})E_{Z'_{1}}\sim_{\mathbb{Q},\,W}0$  for $i=1,2$. 
Then $E_{Z'_{1}}\sim_{\mathbb{Q},\,W}0$, and moreover $D_{Z'_{1}}\sim_{\mathbb{Q},\,W}0$. 
Therefore $K_{X'_{1}}+\Delta_{X'_{1}} \sim_{\mathbb{Q},\,W}0$. 
Since we assume Theorem \ref{thmmain} for $d_{0}-1$, by applying this hypothesis to $f\circ \pi'_{1}:(X'_{1},\Delta_{X'_{1}}) \to W$, we see that $(X'_{1},\Delta_{X'_{1}})$ has a good minimal model. 
Then $(X_{1},\Delta_{X_{1}})$ also has a good minimal model by construction.
Thus we complete the proof. 
\end{step2}
\end{proof}

\begin{rema}\label{remproofcase1}
In the proof of Proposition \ref{propcase1}, we use the condition that $D-e E$ is pseudo-effective from the start. 
On the other hand, we do not use the condition that $D$ is not big until the final part of Step \ref{step5incase1}. 
Therefore we can use the same discussions as in Step \ref{step1incase1}-\ref{step4incase1} and the first half of Step \ref{step5incase1} to prove Theorem \ref{thmmain} in Case \ref{case3}.
\end{rema}

\begin{prop}\label{propcase2}
Fix a positive integer $d_{0}$ and assume Theorem \ref{thmmain} for $d_{0}-1$. 
Let $\pi:(X,\Delta) \to Z$ be as in Theorem \ref{thmmain} satisfying all conditions of Lemma \ref{lemreduction}.
Let $D$ and $E$ be as in the condition (iii) of Lemma \ref{lemreduction}. 

If $D$ is not big and $D-e E$ is not pseudo-effective for any positive rational number $e$, then Theorem \ref{thmmain} holds for $(X,\Delta)\to Z$. 
\end{prop}

\begin{proof} 
We prove it by using similar techniques used in the proof of Proposition \ref{propcase1}. 

\begin{step3}\label{step1incase2}
In this step we construct a diagram used in the proof.  

Let $\{\epsilon_{n}\}_{n\geq1}$ be a strictly decreasing infinite sequence of rational numbers such that $0<\epsilon_{n}<1$ for any $n$ and ${\rm lim}_{n \to \infty}\epsilon_{n}=0$. 
By \cite[Corollary 3.2]{fg-bundle} and the definition of $D$ and $E$, there are $\mathbb{Q}$-divisors $\Psi_{n}$ on $Z$ such that all $(Z,\Psi_{n})$ are klt and $D-\epsilon_{n}E \sim_{\mathbb{Q}}K_{Z}+\Psi_{n}$. 
Fix a sufficiently general ample $\mathbb{Q}$-divisor $A$ on $Z$ such that $D+(1/2)A$ and $(1/2)A-E$ are nef, $(X, \Delta+\pi^{*}A)$ is lc and $(Z,\Psi_{n}+A)$ is klt for any $n$. 
Then 
\begin{equation*}
\begin{split}
K_{Z}+\Psi_{n}+A&\sim_{\mathbb{Q}}D-\epsilon_{n}E+A\\
&=D+\frac{1}{2}A+\epsilon_{n}\Bigl(\frac{1}{2}A-E\Bigr)+\frac{1}{2}(1-\epsilon_{n})A
\end{split}
\end{equation*}
is nef. 
Therefore we can run the $(K_{Z}+\Psi_{n})$-MMP with scaling of $A$, and we get a Mori fiber space 
$$(Z, \Psi_{n}) \dashrightarrow \bigl(Z_{n},(\Psi_{n})_{Z_{n}}\bigr)\overset{f_{n}}{\to}W_{n}.$$   
Let $\tau_{n}=\tau(Z,\Psi_{n}\,;A)$ be the pseudo-effective threshold of $A$ with respect to $(Z,\Psi_{n})$. 
Then $0<\tau_{n}\leq1$. 
By the basic properties of the log MMP with scaling, $D_{Z_{n}}-\epsilon_{n}E_{Z_{n}}+\tau_{n}A_{Z_{n}}\sim_{\mathbb{Q}}K_{Z_{n}}+(\Psi_{n})_{Z_{n}}+\tau_{n}A_{Z_{n}}$ is nef and trivial over $W_{n}$. 
Clearly $D_{Z_{n}}-\epsilon_{n}E_{Z_{n}} \sim_{\mathbb{Q}}K_{Z_{n}}+(\Psi_{n})_{Z_{n}}$ is anti-ample over $W_{n}$ by construction. 
Since $D_{Z_{n}}$ is pseudo-effective, it is nef over $W_{n}$ and hence we see that $E_{Z_{n}}$ is ample over $W_{n}$. 
By Lemma \ref{lemliftmmp}, we obtain the following diagram 
$$
\xymatrix{
(X,\Delta-\epsilon_{n}\Delta'') \ar[d]_{\pi}\ar@{-->}[r]&(X_{n},\Delta_{X_{n}}-\epsilon_{n}\Delta''_{X_{n}})\ar[d]_{\pi_{n}}\!\!\!\!\!\!\!\!\!\!\!\!\!\!\!\\
Z\ar@{-->}[r]&Z_{n} \ar[r]^{f_{n}}&W_{n}
}
$$
such that
\begin{enumerate}
\item[(i)]
the upper horizontal birational map is a sequence of finitely many steps of the $(K_{X}+\Delta-\epsilon_{n}\Delta'')$-MMP,
\item[(ii)]
$\pi_{n}$ is projective and surjective with connected fibers, and
\item[(iii)]
$K_{X_{n}}+\Delta_{X_{n}}\sim_{\mathbb{Q}} \pi_{n}^{*}D_{Z_{n}}$ and $\Delta''_{X_{n}}\sim_{\mathbb{Q}} \pi_{n}^{*}E_{Z_{n}}$ for any $n$.
\end{enumerate} 
\end{step3}

\begin{step3}\label{step2incase2}
In this step we prove that there is an index $n$ such that $(X_{n},\Delta_{X_{n}})$ is lc and $D_{Z_{n}}$ is trivial over $W_{n}$. 
The idea is similar to the proof of \cite[Proposition 8.7]{dhp} or \cite[Lemma 3.1]{gongyo-nonvanishing}. 
By the cone theorem \cite[Theorem 4.5.2]{fujino-book}, $D_{Z_{n}}$ is trivial over $W_{n}$ if and only if $D_{Z_{n}}$ is numerically trivial over $W_{n}$. 

By the same arguments as in Step \ref{step2incase1} in the proof of Proposition \ref{propcase1}, we can find infinitely many indices $n$ such that $(X_{n}, \Delta_{X_{n}})$ is lc. 
Therefore, by replacing $\epsilon_{n}$ with its subsequence, we may assume that $(X_{n}, \Delta_{X_{n}})$ is lc for any $n$. 
Moreover we may assume that the dimension of $W_{n}$ is constant for all $n$ by replacing $\epsilon_{n}$ with its subsequence. 

Since $D_{Z_{n}}$ is nef over $W_{n}$ and since $D_{Z_{n}}-\epsilon_{n}E_{Z_{n}}$ is anti-ample over $W_{n}$, $D_{Z_{n}}-\nu_{n}E_{Z_{n}}$ is numerically trivial over $W_{n}$ for some $0\leq \nu_{n}< \epsilon_{n}$. 
Then we have 
\begin{equation*}
\begin{split}
K_{X_{n}}&+\Delta_{X_{n}}-\nu_{n} \Delta''_{X_{n}}\\
&=K_{X_{n}}+\Delta'_{X_{n}}+(1-\nu_{n}) \Delta''_{X_{n}} \sim_{\mathbb{Q}}\pi_{n}^{*}(D_{Z_{n}}-\nu_{n}E_{Z_{n}})\equiv_{W_{n}}0.
\end{split}
\end{equation*}
Let $F_{n}$ be the general fiber of $f_{n}\circ \pi_{n}$.  
Then $\bigl(F_{n}, (\Delta_{X_{n}}-\nu_{n} \Delta''_{X_{n}})|_{F_{n}}\bigr)$ is lc,  and $\Delta''_{X_{n}}|_{F_{n}}\sim_{\mathbb{Q}}(\pi_{n}^{*}E_{Z_{n}})|_{F_{n}}$ is not numerically trivial since $E_{Z_{n}}$ is ample over $W_{n}$.
Consider 
$$T=\{\nu \in \mathbb{R}_{\geq 0}\,|\, K_{F_{n}}+\Delta'_{X_{n}}|_{F_{n}}+\nu \Delta''_{X_{n}}|_{F_{n}} \equiv 0 {\rm \;\, for\;\, some\;\,}n \}.$$
Clearly $T\supset \{ 1-\nu_{n}\}_{n\geq 1}$ by the definition of $\nu_{n}$.
Conversely, $\nu=1-\nu_{n}$ is the unique number satisfying $K_{F_{n}}+\Delta'_{X_{n}}|_{F_{n}}+\nu \Delta''_{X_{n}}|_{F_{n}} \equiv0$ because $\Delta''_{X_{n}}|_{F_{n}}$ is not numerically trivial.
Therefore we have $T=\{1-\nu_{n}\}_{n\geq 1}$. 
By construction, the dimension of $F_{n}$ is constant for any $n$ and any coefficient of component of $\Delta'_{X_{n}}|_{F_{n}}$ or $\Delta''_{X_{n}}|_{F_{n}}$ is in a finite set which does not depend on $n$. 
Since ${\rm lim}_{n \to \infty}\epsilon_{n}=0$ and $0\leq \nu_{n}< \epsilon_{n}$, $T$ must contain one by the ACC for numerically trivial pairs (cf.~Theorem \ref{thmglobalacc}). 
Then $\nu_{n}=0$ for some $n$. 
In this way, we see that there is an index $n$ such that $(X_{n},\Delta_{X_{n}})$ is lc and $D_{Z_{n}}$ is trivial over $W_{n}$. 
\end{step3}

\begin{step3}\label{step3incase2}
Fix an index $n$ such that $(X_{n},\Delta_{X_{n}})$ is lc and $D_{Z_{n}}$ is trivial over $W_{n}$.
Recall that $K_{Z}+\Psi_{n}\sim_{\mathbb{Q}}D-\epsilon_{n}E$ and $\tau_{n}$ is the pseudo-effective threshold of $A$ with respect to $(Z,\Psi_{n})$.
In this step we prove that $D-t(\epsilon_{n} E-\tau_{n} A)$ is not big for any rational number $0\leq t\leq 1$. 
Note that $D-t(\epsilon_{n} E-\tau_{n} A)$ is pseudo-effective for any $0\leq t\leq 1$ because it is represented by a $\mathbb{Q}_{\geq 0}$-linear combination of $D$ and $D-\epsilon_{n}E+\tau_{n}A$.

Suppose by contradiction that $D-t(\epsilon_{n} E-\tau_{n} A)$ is big for some $0\leq t\leq 1$. 
Then $D_{Z_{n}}-t(\epsilon_{n} E_{Z_{n}}-\tau_{n} A_{Z_{n}})$ is also big. 
On the other hand, $D_{Z_{n}}$ is trivial over $W_{n}$. 
Moreover $D_{Z_{n}}-\epsilon_{n} E_{Z_{n}}+\tau_{n} A_{Z_{n}}$ is also trivial over $W_{n}$ as we mentioned in Step \ref{step1incase2}. 
Therefore $D_{Z_{n}}-t(\epsilon_{n} E_{Z_{n}}+\tau_{n} A_{Z_{n}})$ is trivial over $W_{n}$ and thus it is $\mathbb{Q}$-linearly equivalent to the pullback of a $\mathbb{Q}$-divisor on $W_{n}$. 
Because ${\rm dim}\,W_{n}<{\rm dim}\,Z_{n}$, we get a contradiction. 
Therefore $D-t(\epsilon_{n} E-\tau_{n} A)$ is not big for any rational number $0\leq t\leq 1$.

For this $n$, we put $\epsilon=\epsilon_{n}$ and $\tau=\tau_{n}$ in the rest of the proof. 
Since $K_{X}+\Delta-t(\epsilon\Delta''-\tau \pi^{*}A) \sim_{\mathbb{Q}}\pi^{*}\bigl(D-t(\epsilon E-\tau A)\bigr)$, $K_{X}+\Delta-t(\epsilon\Delta''-\tau \pi^{*}A)$ is also pseudo-effective for any $0\leq t\leq 1$.
\end{step3}

\begin{step3}\label{step4incase2}
We set 
$$\widetilde{E}=E-\frac{\tau}{\epsilon} A \qquad {\rm and} \qquad \widetilde{\Delta}=\Delta''-\frac{\tau}{\epsilon} \pi^{*}A\sim_{\mathbb{Q}}\pi^{*}\widetilde{E}.$$
Note that $\widetilde{E}$ and $\widetilde{\Delta}$ may not be effective. 
We see that $D-t\widetilde{E}$ is pseudo-effective for any $0\leq t \leq \epsilon$ because $D-t\widetilde{E}=D-(t/\epsilon)(\epsilon E-\tau A)$. 
Since $K_{X}+\Delta-t\widetilde{\Delta} \sim_{\mathbb{Q}}\pi^{*}(D-t\widetilde{E})$, $K_{X}+\Delta-t\widetilde{\Delta}$ is also pseudo-effective for any $0\leq t \leq \epsilon$.
Moreover, for any $0\leq t\leq \epsilon$, $(X,\Delta-t\widetilde{\Delta})$ is lc. 
To see this, recall that $0< \tau \leq 1$, which is mentioned in Step \ref{step1incase2}. 
So we have $0 \leq t\tau/\epsilon\leq 1$ for any $0\leq t\leq \epsilon$. 
We also recall that $(X,\Delta+\pi^{*}A)$ is lc.
Since $\Delta-t\widetilde{\Delta}=\Delta-t\Delta''+(t\tau/\epsilon)\pi^{*}A$, the pair $(X,\Delta-t\widetilde{\Delta})$ is indeed lc for any $0\leq t\leq \epsilon$.
\end{step3}

\begin{step3}\label{step5incase2}
From this step we use the same arguments as in the proof of Proposition \ref{propcase1}. 
We only write down the outline of the proof. 

Fix a strictly decreasing infinite sequence of rational numbers $\{a_{n}\}_{n\geq1}$ such that $0<a_{n}<\epsilon$ for any $n$ and ${\rm lim}_{n \to \infty}a_{n}=0$. 
With the divisors $D$ and $\widetilde{E}$, we carry out the same arguments as in Step \ref{step1incase1} in the proof of Proposition \ref{propcase1}. 
When we apply the arguments of Step \ref{step2incase1} in the proof of Proposition \ref{propcase1}, a minor change is needed. 
More precisely, we need to carry out the arguments with the effective part of $\widetilde{\Delta}$. 
But we can eventually obtain the same result, that is, $(X_{n},\Delta_{X_{n}})$ is lc for infinitely many indices $n$.  
By the same arguments as in Step \ref{step3incase1} in the proof of Proposition \ref{propcase1}, replacing $\{a_{n}\}_{n\geq1}$ with its subsequence, we get a strictly decreasing infinite sequence $\{a_{n}\}_{n\geq1}$ of positive rational numbers such that 
\begin{enumerate}
\item[(i)]
$a_{n}<\epsilon$ for any $n$ and ${\rm lim}_{n\to \infty}a_{n}=0$, and 
\item[(ii)]
for any $n\geq1$, there is a diagram
$$
\xymatrix@C=25pt{
\bigl(X,\Delta-a_{n}\widetilde{\Delta}\bigr) \ar[d]_{\pi}\ar@{-->}[r]^{\!\!\!\!\!\!\!\!\!\!\!\!\phi_{n}}&\bigl(X_{n},\Delta_{X_{n}}-a_{n}\widetilde{\Delta}_{X_{n}}\bigr)\ar[d]_{\pi_{n}}\\
Z\ar@{-->}[r]&Z_{n}
}
$$
such that 
\begin{enumerate}
\item[(ii-a)]
$X_{n}$ and $Z_{n}$ are $\mathbb{Q}$-factorial, $(X_{n},\Delta_{X_{n}})$ is lc, $(X_{n},0)$ is klt and $\pi_{n}$ is a projective surjective morphism with connected fibers,
\item[(ii-b)]
$K_{X_{n}}+\Delta_{X_{n}} \sim_{\mathbb{Q}}\pi_{n}^{*}D_{Z_{n}}$ and $\widetilde{\Delta}_{X_{n}}\sim_{\mathbb{Q}}\pi_{n}^{*}\widetilde{E}_{Z_{n}}$, and 
\item[(ii-c)]
$\phi_{n}$ is a sequence of finitely many steps of the $(K_{X}+\Delta-a_{n}\widetilde{\Delta})$-MMP to a good minimal model $(X_{n},\Delta_{X_{n}}-a_{n}\widetilde{\Delta}_{X_{n}})$.
\end{enumerate} 
\end{enumerate} 
Now we carry out the argument as in the latter part of Step \ref{step3incase1} in the proof of Proposition \ref{propcase1} with $(X,\Delta)$ and $\widetilde{\Delta}$ instead of $(X,\Delta)$ and $\Delta''$. 
Then we see that we can assume
\begin{enumerate}
\item[(ii-d)]
$X_{i}$ and $X_{j}$ are isomorphic in codimension one for any $i$ and $j$.
\end{enumerate}
Moreover, as in the last part of Step \ref{step3incase1} in the proof of Proposition \ref{propcase1}, we see that $D_{Z_{n}}-(a_{n}-\delta)\widetilde{E}_{Z_{n}}$ is not big for any sufficiently small rational number $\delta >0$. 
\end{step3}

\begin{step3}\label{step6incase2}
Finally we complete the proof by using the same arguments as in Step \ref{step4incase1} and Step \ref{step5incase1} in the proof of Proposition \ref{propcase1}. 
To carry out, we only check that every lc center of $(X_{1},\Delta_{X_{1}}-t\widetilde{\Delta}_{X_{1}})$ dominates $Z_{1}$ for any $0<t \leq a_{1}$. 
Once we can check this, we see that $(X_{1},\Delta_{X_{1}})$ has a good minimal model by the same arguments as in Step \ref{step5incase1} in the proof of Proposition \ref{propcase1}, and thus $(X,\Delta)$ has a weak lc model with semi-ample log canonical divisor by the same arguments as in Step \ref{step4incase1} in the proof of Proposition \ref{propcase1}. 

For any  $0<t \leq a_{1}$, every lc center of $(X_{1},\Delta_{X_{1}}-t\widetilde{\Delta}_{X_{1}})$ is also an lc center of $(X_{1},\Delta_{X_{1}}-a_{1}\widetilde{\Delta}_{X_{1}})$ because $(X_{1},\Delta_{X_{1}})$ is lc. 
Therefore we may check the condition only when $t=a_{1}$. 

Recall again that $\tau$ is a rational number such that $0< \tau \leq 1$. 
Since $a_{1}<\epsilon$, we have $a_{1}\tau/\epsilon<1$. 
Since $(X,\Delta+\pi^{*}A)$ is lc and
$$\Delta-a_{1}\widetilde{\Delta}=\Delta-a_{1}\Delta''+(a_{1}\tau/\epsilon) \pi^{*}A\leq\Delta+\pi^{*}A,$$
every lc center of $(X,\Delta-a_{1}\widetilde{\Delta})$ is an lc center of $(X,\Delta+\pi^{*}A)$, and moreover it is also an lc center of $(X,\Delta-\Delta'')$. 
Since any lc center of $(X,\Delta-\Delta'')$ dominates $Z$ by the condition (iii) of Lemma \ref{lemlogresol}, we see that any lc center of $(X,\Delta-a_{1}\widetilde{\Delta})$ dominates $Z$. 
Since $\phi$ is a sequence of finitely many steps of the $(K_{X}+\Delta-a_{1}\widetilde{\Delta})$-MMP, any lc center of $(X_{1},\Delta_{X_{1}}-a_{1}\widetilde{\Delta}_{X_{1}})$ dominates $Z_{1}$. 
Thus we complete the proof.
\end{step3}
\end{proof}

Finally we prove Theorem \ref{thmmain} in Case \ref{case3}. 
As we state in Proposition \ref{propcase3} below, we can in fact prove the case with assumptions weaker than Proposition \ref{propcase1} or Proposition \ref{propcase2}, i.e., we can prove the case without assuming the existence of a good minimal model or a Mori fiber space for all $d_{0}$-dimensional projective Kawamata log terminal pairs with boundary $\mathbb{Q}$-divisors. 

\begin{prop}\label{propcase3}
Fix a positive integer $d_{0}$. 
Assume Theorem \ref{thmmain} for $d_{0}-1$, and assume the existence of a good minimal model or a Mori fiber space for all $d$-dimensional projective Kawamata log terminal pairs  with boundary $\mathbb{Q}$-divisors such that $d\leq d_{0}-1$. 

Let $\pi:(X,\Delta) \to Z$ be as in Theorem \ref{thmmain} satisfying all conditions of Lemma \ref{lemreduction}.
Let $D$ be as in the condition (iii) of Lemma \ref{lemreduction}. 

If $D$ is big, then $(X,\Delta)$ has a good minimal model. 
\end{prop}

\begin{proof}
Let $E$ be as in the condition (iii) of Lemma \ref{lemreduction}, that is, an effective $\mathbb{Q}$-divisor such that $\Delta'' \sim_{\mathbb{Q}}\pi^{*}E$. 
We may assume that $E\neq0$ because otherwise the proposition follows from Proposition \ref{propbigkltcase}. 
Fix a sufficiently small positive rational number $e < 1$ such that $D-eE$ is big. 
We prove the proposition with several steps.

\begin{step}\label{step1incase3}
First we note that the arguments of the proof of Proposition \ref{propcase1} work with some minor changes by using Proposition \ref{propbigkltcase} (cf.~Remark \ref{remproofcase1}). 
Therefore we only have to prove that there is a good minimal model of $(X,\Delta)$ under the assumption that $D-aE$ and $D-(a-u)E$ are semi-ample, where $0<a<e$ and $0<u\ll a$ are rational numbers. 
Note that $D-aE$ is big since $a<e$.

Pick a sufficiently large and divisible positive integer $m$ such that $a/(m+1)<u$ and $1/m<u$. 
Fix $A\sim_{\mathbb{Q}}m(D-aE)$ a general semi-ample $\mathbb{Q}$-divisor. 
Then $A$ is big and we have
$$A+E \sim_{\mathbb{Q}}m(D-aE)+E=m\Bigl(D-\Bigl(a-\frac{1}{m}\Bigr)E\Bigr)$$
and 
$$D+A\sim_{\mathbb{Q}}D+m(D-aE)=(m+1)\Bigl(D-\Bigl(a-\frac{a}{m+1}\Bigr)E\Bigr).$$
Since $0<a/(m+1)<u$, $0<1/m<u$ and $D-aE$ and $D-(a-u)E$ are semi-ample, we see that $A+E$ and $D+A$ are semi-ample. 
By Lemma \ref{lemd-mmp}, there is a sequence of birational  maps of the $D$-MMP with scaling of $A$ 
$$(Z=Z_{0},\lambda_{0}) \dashrightarrow \cdots \dashrightarrow (Z_{i},\lambda_{i})\dashrightarrow \cdots$$
such that ${\rm lim}_{i \to \infty} \lambda_{i}=0$. 

First we prove existence of a log minimal model of $(X,\Delta)$. 
Take a dlt blow-up $f:(Y,\Gamma) \to (X,\Delta)$.
Then we only have to prove that $(Y,\Gamma)$ has a log minimal model.
Set $g_{0}=\pi\circ f:Y\to Z$ and $A'=g_{0}^{*}A$. 
By construction we have $K_{Y}+\Gamma\sim_{\mathbb{Q}}g_{0}^{*}D$. 
Since $A$ is a general semi-ample divisor on $Z$, we may assume that $(Y,\Gamma+A')$ is also dlt.
Set $G=f^{*}\Delta''$. 
By Lemma \ref{lemliftmmp}, we have the following diagram
$$
\xymatrix@C=12pt{
(Y=Y_{0},\Gamma=\Gamma_{0}) \ar_{g_{0}}[d] \ar@{-->}[r]&
\cdots \ar@{-->}[r]&(Y_{k_{1}},\Gamma_{Y_{k_{1}}})\ar_{g_{1}}[d] \ar@{-->}[r]&
\cdots\ar@{-->}[r]& (Y_{k_{i}},\Gamma_{Y_{k_{i}}}) \ar_{g_{i}}[d] \ar@{-->}[r]&
\cdots \\
(Z_{0},\lambda_{0})\ar@{-->}[rr]&&(Z_{1},\lambda_{1})\ar@{-->}[rr]&&(Z_{i},\lambda_{i})\ar@{-->}[r]&
\cdots
}
$$
such that 
\begin{enumerate}
\item[(i)]
the upper horizontal sequence of birational maps is a sequence of the $(K_{Y}+\Gamma)$-MMP with scaling of $A'$,
\item[(ii)]
if we set $k_{0}=0$ and 
$$\lambda'_{j}={\rm inf}\{\mu \in \mathbb{R}_{\geq0} \,|\, K_{Y_{j}}+\Gamma_{Y_{j}}+\mu A'_{Y_{j}}{\rm \;is \;nef}\},$$
then $\lambda'_{j}=\lambda_{i}$ for any $i \geq 0$ and $k_{i}\leq j<k_{i+1}$, and
\item[(iii)] 
$K_{Y_{k_{i}}}+\Gamma_{Y_{k_{i}}}\sim_{\mathbb{Q}}g_{i}^{*}D_{Z_{i}}$ and $G_{Y_{k_{i}}}\sim_{\mathbb{Q}}g_{i}^{*}E_{Z_{i}}$.
\end{enumerate}
\end{step}

\begin{step}\label{step2incase3}
In this step and the next step, we prove that the $(K_{Y}+\Gamma)$-MMP with scaling of $A'$ terminates. 

Let $C$ be any curve on $Y_{j}$ contracted by the extremal contraction associated to $Y_{j}\dashrightarrow Y_{j+1}$. 
In this step we prove that $C\subset {\rm Supp}\,G_{Y_{j}}$. 
If we can check this, we may prove that the above $(K_{Y}+\Gamma)$-MMP occurs eventually disjoint from ${\rm Supp}\,G$.

By the definition of the log MMP with scaling, $C\cdot (K_{Y_{j}}+\Gamma_{Y_{j}})<0$ and $C\,\cdot\, (K_{Y_{j}}+\Gamma_{Y_{j}}+\lambda'_{j}A'_{Y_{j}})=0$. 
Therefore $(C\cdot A'_{Y_{j}})>0$. 
We also have $A'\sim_{\mathbb{Q}}m(K_{Y}+\Gamma-aG)$ by the definition of $A$. 
Then 
$$a(C\,\cdot\, G_{Y_{j}})=C\cdot (K_{Y_{j}}+\Gamma_{Y_{j}})-\frac{1}{m} (C\cdot A'_{Y_{j}})<0.$$
Since $a>0$ and $G_{Y_{j}}$ is effective, we see that $C\subset {\rm Supp}\,G_{Y_{j}}$.
\end{step}

\begin{step}\label{step3incase3}
We apply the standard arguments of the special termination (cf.~\cite{fujino-sp-ter}). 
Note that ${\rm Supp}\,G\subset {\rm Supp}\,\llcorner \Gamma \lrcorner$. 
By replacing $(Y,\Gamma)$ with $(Y_{k_{i}},\Gamma_{Y_{k_{i}}})$ for some $i\gg0$, we may assume that the $(K_{Y}+\Gamma)$-MMP contains only flips and flipping locus on each flip contains no lc centers. 

Let $S\subset {\rm Supp}\,G$ be an lc center of $(Y,\Gamma)$ and let $S_{j}$ be the birational transform of $S$ on $Y_{j}$. 
We define $\mathbb{Q}$-divisors  $\Gamma_{S_{j}}$ by the adjunction $(K_{Y_{j}}+\Gamma_{Y_{j}})|_{S_{j}}=K_{S_{j}}+\Gamma_{S_{j}}$. 
Then $(S_{j},\Gamma_{S_{j}})$ is dlt. 
By induction on the dimension of $S$ we show that  for any $j\gg0$ the induced birational map $\phi_{j}:(S_{j},\Gamma_{S_{j}})\dashrightarrow (S_{j+1},\Gamma_{S_{j+1}})$ is an isomorphism. 
By the argument as in \cite{fujino-sp-ter}, for any $j\gg 0$, $S_{j}$ and $S_{j+1}$ are isomorphic in codimension one and $\phi_{j*}(K_{S_{j}}+\Gamma_{S_{j}})=K_{S_{j+1}}+\Gamma_{S_{j+1}}$. 
By replacing $(Y,\Gamma)$ with $(Y_{k_{i}},\Gamma_{k_{i}})$ for some $i\gg0$, we may assume that $S_{j}$ satisfies the above properties for any $j\geq0$. 
Let $(T,\Theta)\to (S,\Gamma_{S})$ be a dlt blow-up and $A''$ be the pullback of $A'$. 
By replacing $A''$ if necessary, we may assume that $A''$ is effective and $(T,\Theta+A'')$ is dlt.
Set $T_{0}^{0}=T$ and $\Theta_{T_{0}^{0}}=\Theta$. 
By the same arguments as in the proof of Lemma \ref{lemliftmmp} (see also \cite{fujino-sp-ter}), we get the following diagram 
$$
\xymatrix@C=12pt{
(T_{0}^{0},\Theta_{T_{0}^{0}}) \ar[d] \ar@{-->}[r]&
\cdots \ar@{-->}[r]&(T^{j}_{i},\Theta_{T^{j}_{i}}) \ar@{-->}[r]& \cdots\ar@{-->}[r]& (T^{l_{i}}_{i}=T_{i+1}^{0},\Theta_{T_{i+1}^{0}}) \ar[d] \ar@{-->}[r]&\cdots \\
(Y,\Gamma,\lambda'_{0})\ar@{-->}[rr]&&\cdots \ar@{-->}[rr]&&(Y_{i+1},\Gamma_{Y_{i+1}},\lambda'_{i+1})\ar@{-->}[r]&\cdots
}
$$
such that 
\begin{enumerate}
\item[(i)]
the upper horizontal sequence of birational maps is a sequence of steps of the $(K_{T}+\Theta)$-MMP with scaling $A''$, and
\item[(ii)] 
the morphism $T_{i}^{0}\to Y_{i}$ is the composition of a dlt blow-up of $(S_{i},\Gamma_{S_{i}})$ and the inclusion $S_{i}\hookrightarrow Y_{i}$. 
\end{enumerate}
By construction, we also have the following property.
\begin{enumerate} 
\item[(iii)]
If we set 
$$\lambda^{j}_{i}={\rm inf}\{\mu \in \mathbb{R}_{\geq0} \,|\, K_{T_{i}^{j}}+\Theta_{T^{j}_{i}}+\mu A''_{T_{i}^{j}}{\rm \;is \;nef},\;0\leq j<l_{i}\},$$
then $\lambda^{0}_{i} \leq \lambda'_{i}$.
\end{enumerate} 
Note that we may have $\lambda^{0}_{i} < \lambda'_{i}$ because the morphism $T_{i}^{0}\to Y_{i}$ is not surjective. 
If $\lambda^{0}_{i} < \lambda'_{i}$, then we have $T_{i}^{0} \simeq T_{i}^{1} \simeq \cdots \simeq T_{i}^{l_{i}}=T_{i+1}^{0}$ by construction. 

By (iii) of the above properties, $K_{T_{i}^{0}}+\Theta_{T_{i}^{0}}+\lambda_{i}^{0}A''_{T_{i}^{0}}$ is pseudo-effective. 
Then $K_{T}+\Theta+\lambda_{i}^{0}A''$ is also pseudo-effective. 
Since ${\rm lim}_{i \to \infty} \lambda'_{i}=0$, we see that $K_{T}+\Theta$ is pseudo-effective. 
Now consider the composition of morphisms $T \to S \hookrightarrow Y \to Z,$ which we denote $h:T \to Z$. 
Recall that $S\subset {\rm Supp}\,G$ and that $g_{0}({\rm Supp}\,G)\subsetneq Z$.  
Let $Z_{T}$ be the normalization of  $h(T)$. 
Then $K_{T}+\Theta \sim_{\mathbb{Q},\,Z_{T}}0$ because $K_{Y}+\Gamma_{Y}\sim_{\mathbb{Q},\,Z}0$. 
Moreover ${\rm dim}\,Z_{T}<{\rm dim}\,Z$ by construction.  
Since we assume Theorem \ref{thmmain} for $d_{0}-1$, applying the hypothesis to $(T,\Theta)\to Z_{T}$, we see that $(T,\Theta)$ has a good minimal model. 
Then the above $(K_{T}+\Theta)$-MMP with scaling 
terminates by \cite[Theorem 4.1 (iii)]{birkar-flip} because ${\rm lim}_{i \to \infty} \lambda'_{i}=0$. 
Therefore the induced birational map $\phi_{j}:(S_{j},\Gamma_{S_{j}})\dashrightarrow (S_{j+1},\Gamma_{S_{j+1}})$ is an isomorphism for any $j\gg0$. 
Then, by the argument of the special termination (cf.~\cite{fujino-sp-ter}), we see that the $(K_{Y}+\Gamma)$-MMP with scaling occurs eventually disjoint from ${\rm Supp}\,G$.
In this way we see that the $(K_{Y}+\Gamma)$-MMP with scaling of $A'$ constructed in Step \ref{step1incase3} must terminate.
\end{step}

\begin{step}\label{step4incase3}
Finally we prove that $(X,\Delta)$ has a good minimal model. 
By Step \ref{step3incase3}, the $D$-MMP with scaling of $A$ constructed in Step \ref{step1incase3} terminates (cf.~Lemma \ref{lemliftmmp}). 
By Lemma \ref{lemliftmmp}, we have the following diagram 
$$
\xymatrix@C=21pt{
(X,\Delta) \ar_{\pi}[d] \ar@{-->}[r]&\cdots\ar@{-->}[r]& (X_{k_{i}},\Delta_{X_{k_{i}}}) \ar_{\pi_{i}}[d] \\
(Z, \lambda_{0}) \ar@{-->}[r]&\cdots\ar@{-->}[r]&(Z_{i}, \lambda_{i}=0)
}
$$
such that $(X_{k_{i}},\Delta_{X_{k_{i}}})$ is a log minimal model of $(X,\Delta)$. 
We note that $K_{X_{k_{i}}}+\Delta_{X_{k_{i}}}\sim_{\mathbb{Q}}\pi_{i}^{*}D_{Z_{i}}$ and $D_{Z_{i}}$ is big. 
Then $K_{X_{k_{i}}}+\Delta_{X_{k_{i}}}$ is semi-ample by Lemma \ref{lemabundance} below. 
So we are done.  
\end{step}
\end{proof}

\begin{lemm}\label{lemabundance}
Fix a positive integer $d_{0}$. 
Assume Theorem \ref{thmmain} for $d_{0}-1$, and assume the existence of a good minimal model or a Mori fiber space for all $d$-dimensional projective Kawamata log terminal pairs with boundary $\mathbb{Q}$-divisors such that $d\leq d_{0}-1$. 

Let $\pi:X \to Z$ be a projective surjective morphism of a normal projective varieties such that ${\rm dim}\,Z\leq d_{0}$, and $(X,\Delta)$ be a log canonical pair such that $\Delta$ is a $\mathbb{Q}$-divisor. 
Suppose that $K_{X}+\Delta$ is nef and $K_{X}+\Delta\sim_{\mathbb{Q}}\pi^{*}D$ for a big $\mathbb{Q}$-Cartier $\mathbb{Q}$-divisor $D$ on $Z$. 

Then $K_{X}+\Delta$ is semi-ample. 
\end{lemm}

\begin{proof}
By taking a dlt blow-up, we can assume $(X,\Delta)$ is $\mathbb{Q}$-factorial dlt. 
We show $K_{X}+\Delta$ is nef and log abundant. 
Indeed, $K_{X}+\Delta$ is nef and abundant since it is $\mathbb{Q}$-linearly equivalent to the pullback of a nef and  big $\mathbb{Q}$-divisor on $Z$.
Let $T$ be any lc center of $(X,\Delta)$ and define $\Delta_{T}$ by the adjunction $(K_{X}+\Delta)|_{T}=K_{T}+\Delta_{T}$. 
Then $(T,\Delta_{T})$ is dlt and $K_{T}+\Delta_{T}$ is nef. 
Let $Z_{T}$ be the normalization of the image of $T$ on $Z$. 
If ${\rm dim}\,Z_{T}<{\rm dim}\,Z$, by Theorem \ref{thmmain} for $d_{0}-1$, we see that $K_{T}+\Delta_{T}$ is semi-ample. 
In particular it is nef and abundant. 
On the other hand, if ${\rm dim}\,Z_{T}={\rm dim}\,Z$, it is easy to see that $K_{T}+\Delta_{T}$ is $\mathbb{Q}$-linearly equivalent to the pullback of nef and big $\mathbb{Q}$-divisor on $Z_{T}$. 
Therefore it is nef and abundant. 
Thus we see that $K_{X}+\Delta$ is nef and log abundant. 
Then $K_{X}+\Delta$ is semi-ample by \cite[Theorem 4.12]{fujino-gongyo}. 
So we are done. 
\end{proof}

\section{Proof of other results} \label{secproofcorollary}

In this section we prove Theorem \ref{thmtrueiitaka} and Corollary \ref{corlcring}. 

\begin{theo}\label{thmfibration}
Let $\pi:X \to Z$ be a projective surjective morphism of normal projective varieties and let $(X,\Delta)$ be a log canonical pair such that $\Delta$ is a $\mathbb{Q}$-divisor. 
Suppose that $K_{X}+\Delta \sim_{\mathbb{Q}}\pi^{*}D$ for a $\mathbb{Q}$-Cartier $\mathbb{Q}$-divisor $D$ on $Z$.  

If ${\rm dim}\,Z \leq 3$ or ${\rm dim}\,Z = 4$ and $D$ is big, then $(X,\Delta)$ has a good minimal model or a Mori fiber space. 
\end{theo}

\begin{proof}
By the same arguments as in the proof of Lemma \ref{lemreduction}, we can assume that $\pi:(X,\Delta)\to Z$ satisfies all the conditions of Lemma \ref{lemreduction}. 
Note that bigness of $D$ still holds after the process. 
Since the log MMP and the abundance conjecture hold for all log canonical threefolds, the theorem follows from Theorem \ref{thmmain} and Proposition \ref{propcase3}. 
\end{proof}

\begin{proof}[Proof of Theorem \ref{thmtrueiitaka}]
Let $f:X\dashrightarrow W$ be the Iitaka fibration. 
Taking an appropriate resolution of $X$ if necessary, we may in particular assume that $X$ is $\mathbb{Q}$-factorial, $(X,0)$ is klt and $f$ is a morphism. 
By \cite[Theorem 2]{kawamata} (see also \cite[Theorem 0.3]{ak}), we can in particular assume that $W$ is smooth and all fibers have the same dimension (cf.~\cite[Theorem 2.1]{haconxu-lcc}). 
By construction, there is an effective $\mathbb{Q}$-divisor $E$ such that $K_{X}+\Delta \sim_{\mathbb{Q}}E$. 
Then we can write $E=E^{h}+E^{v}$, where every component of $E^{h}$ dominates $W$ and $E^{v}$ is vertical. 
Since all fibers of $f$ have the same dimension, the image of any component of $E^{v}$ on $W$ is a divisor. 
Then we can consider 
$$\mu_{B}={\rm sup}\{\mu \,|\, E^{v}-\mu f^{*}B {\rm \,\; is \,\; effective}\}$$
for any prime divisor $B$ on $W$. 
Then it is easy to see that $\mu_{B}$ is a rational number for any $B$ and there are only finitely many divisors $B$ such that $\mu_{B}>0$. 
Set 
$$B' =\sum_{B}\mu_{B}B \qquad {\rm and} \qquad E'=E^{v}-f^{*}B'.$$ 
Then $K_{X}+\Delta \sim_{\mathbb{Q},\,W}E^{h}+E'$ and $E'$ is effective. 
Moreover we see that $E'$ is very exceptional over $W$ (cf.~\cite[Definition 3.1]{birkar-flip}). 

We run the $(K_{X}+\Delta)$-MMP over $W$ with scaling of an ample divisor 
$$X=X_{0} \dashrightarrow X_{1}\dashrightarrow \cdots \dashrightarrow X_{i}\dashrightarrow \cdots.$$ 
Let $f_{i}:X_{i}\to W$ be the induced morphism and let $F_{i}$ be the general fiber of $f_{i}$. 
Recall that $(F,\Delta_{F})$ has a good minimal model by the hypothesis. 
Since $\kappa \bigl(F_{i},(K_{X_{i}}+\Delta_{X_{i}})|_{F_{i}}\bigr)=0$, we have 
$$(K_{X_{i}}+\Delta_{X_{i}})|_{F_{i}} \sim_{\mathbb{Q}} E_{X_{i}}|_{F_{i}} \sim_{\mathbb{Q}} (E^{h}_{X_{i}}+E'_{X_{i}})|_{F_{i}}  \sim_{\mathbb{Q}} 0$$
for any $i \gg 0$. 
Therefore $E^{h}_{X_{i}}+E'_{X_{i}}$ is vertical and thus we have $K_{X_{i}}+\Delta_{X_{i}} \sim_{\mathbb{Q},\,W}E'_{X_{i}}$. 
We note that $E'_{X_{i}}$ is very exceptional over $W$ because the $(K_{X}+\Delta)$-MMP occurs only in ${\rm Supp}\,(E^{h}+E')$. 
Moreover $K_{X_{i}}+\Delta_{X_{i}} \sim_{\mathbb{Q},\,W}E'_{X_{i}}$ is the limit of movable divisors over $W$ for any $i \gg 0$. 
Then $E'_{X_{i}}=0$ by \cite[Lemma 3.3]{birkar-flip}. 
Therefore $K_{X_{i}}+\Delta_{X_{i}} \sim_{\mathbb{Q},\,W}0$ for some $i$. 
Let $D$ be a $\mathbb{Q}$-divisor on $W$ such that $K_{X_{i}}+\Delta_{X_{i}} \sim_{\mathbb{Q}}f_{i}^{*}D$. 
Then $D$ is big since $\kappa(W,D)=\kappa(X_{i},K_{X_{i}}+\Delta_{X_{i}})={\rm dim}\,W$. 

If $(X,\Delta)$ is klt, then $(X_{i},\Delta_{X_{i}})$ is also klt and  $(X_{i},\Delta_{X_{i}})$ has a good minimal model by Proposition \ref{propbigkltcase}. 
Note that $W$ is in particular $\mathbb{Q}$-factorial from our assumption. 
On the other hand, if $\kappa(X,K_{X}+\Delta)\leq 4$, then ${\rm dim}\,W\leq 4$ and therefore $(X_{i},\Delta_{X_{i}})$ has a good minimal model by Theorem \ref{thmfibration}. 

Therefore we see that $(X,\Delta)$ has a good minimal model. 
\end{proof}

\begin{proof}[Proof of Corollary \ref{corlcring}]
Since $(X,\Delta)$ is not of log general type, we have $\kappa(X,K_{X}+\Delta)\leq 4$. 
We take the Iitaka fibration $f:X\dashrightarrow W$, and taking a resolution we can assume that $f$ is a morphism. 
We can assume $\kappa(X,K_{X}+\Delta)>0$ because otherwise the statement is obvious. 

Suppose that $\kappa(X,K_{X}+\Delta)=1$. 
In general, there is a $\mathbb{Q}$-divisor $B$ on $W$ and a positive integer $m$ such that 
$\mathcal{R}\bigl(X,m(K_{X}+\Delta)\bigr)\simeq \mathcal{R}(W,mB)$ (cf.~\cite{fujino-mori}). 
Since $W$ is a smooth curve, $\mathcal{R}(W,mB)$ is finitely generated. 
Then $\mathcal{R}\bigl(X,m(K_{X}+\Delta)\bigr)$ is finitely generated, and thus $\mathcal{R}(X,K_{X}+\Delta)$ is finitely generated. 

Suppose that $\kappa(X,K_{X}+\Delta) \geq2$. 
Let $F$ be the general fiber of the Iitaka fibration and let $(F,\Delta_{F})$ be the restriction of $(X,\Delta)$. 
Then $(F,\Delta_{F})$ is lc and ${\rm dim}\,F \leq3$ by construction, and thus $(F,\Delta_{F})$ has a good minimal model. 
Then $(X,\Delta)$ has a good minimal model by Theorem \ref{thmtrueiitaka}. 
Thus $\mathcal{R}(X,K_{X}+\Delta)$ is finitely generated. 
\end{proof}



\begin{thebibliography}{BCHM}

\bibitem[AK]{ak}
D.~Abramovich, K.~Karu, 
Weak semistable reduction in characteristic $0$, 
Invent. math. {\textbf{139}} (2000), no. 2, 241--273. 


\bibitem[A]{ambro}
F.~Ambro, 
The moduli $b$-divisor of an lc trivial fibration, 
Compos. Math. {\textbf{141}} (2005), no. 2, 385--403. 


\bibitem[B1]{birkar-existII}
C.~Birkar, 
On existence of log minimal models I\!I, 
J. Reine Angew Math. {\textbf{658}} (2011), 99--113.


\bibitem[B2]{birkar-flip}
C.~Birkar, 
Existence of log canonical flips and a special LMMP, 
Publ. Math. Inst. Hautes \'Etudes Sci. {\textbf{115}} (2012), no. 1, 325--368.


\bibitem[BCHM]{bchm}
C.~Birkar, P.~Cascini, C.~D.~Hacon, J.~M\textsuperscript{c}Kernan, 
Existence of minimal models for varieties of log general type, 
J. Amer. Math. Soc. {\textbf{23}} (2010), no. 2, 405--468.

\bibitem[BH]{birkarhu-arg}
C.~Birkar, Z.~Hu,
Log canonical pairs with good augmented base loci,
Compos. Math. {\textbf{150}} (2014), no. 4, 579--592. 



\bibitem[DHP]{dhp}
J.-P. Demailly, C.~D.~Hacon, M.~P\u{a}un, 
Extension theorems, non-vanishing and the existence of good minimal models, 
Acta Math. {\textbf{210}} (2013), no. 2, 203--259.


\bibitem[F1]{fujino-sp-ter}
O.~Fujino, 
{\it Special termination and reduction to pl flips.} In Flips for $3$-folds and $4$-folds, 
Oxford University Press (2007).

\bibitem[F2]{fujino-lcring}
O.~Fujino, 
Finite generation of the log canonical ring in dimension four, 
Kyoto J. Math. {\textbf{50}} (2010), no. 4, 671--684.

\bibitem[F3]{fujino-fund}
O.~Fujino, 
Fundamental theorems for the log minimal model program, 
Publ. Res. Inst. Math. Sci. {\textbf{47}} (2011), no. 3, 727--789. 




\bibitem[F4]{fujino-book}
O.~Fujino, 
{\em Foundations of the minimal model program}, 
MSJ Mem. \textbf{35}, Mathematical Society in Japan, Tokyo, 2017. 

\bibitem[F5]{fujino-remarks}
O.~Fujino, 
Some remarks on the minimal model program for log canonical pairs, 
J. Math. Sci. Univ. Tokyo {\textbf{22}} (2015), no. 1, 149--192. 


\bibitem[FG1]{fg-bundle}
O.~Fujino, Y.~Gongyo, 
On canonical bundle formulas and subadjunctions, 
Michigan Math. J. {\textbf{61}} (2012), no. 2, 255-264. 


\bibitem[FG2]{fujino-gongyo}
O.~Fujino, Y.~Gongyo, 
Log pluricanonical representations and abundance conjecture, 
Compos. Math. {\textbf{150}} (2014) no. 4, 593--620.


\bibitem[FG3]{fg-lctrivial}
O.~Fujino, Y.~Gongyo, 
On the moduli b-divisors of lc-trivial fibrations,
Ann. Inst. Fourier {\textbf{64}} (2014), no. 4, 1721--1735. 


\bibitem[FG4]{fg-lcring}
O.~Fujino, Y.~Gongyo, 
On log canonical rings, 
to appear in Kawamata 60.


\bibitem[FM]{fujino-mori}
O.~Fujino, S.Mori, 
A canonical bundle formula, J. Differential Geom. {\textbf{56}} (2000), no. 1, 167--188.






\bibitem[G]{gongyo-nonvanishing}
Y.~Gongyo, 
Remarks on the non-vanishing conjecture, 
Adv. Stud. Pure Math. {\textbf{65}} (2015), 
Algebraic geometry in East Asia--Taipei 2011, 107--116.


\bibitem[GL]{gongyolehmann}
Y.~Gongyo, B.~Lehmann, 
Reduction maps and minimal model theory, 
Compos. Math. {\textbf{149}} (2013), no. 2, 295--308.



\bibitem[HMX]{hmx-acc}
C.~D.~Hacon, J.~M\textsuperscript{c}Kernan, C.~Xu,
ACC for log canonical thresholds, 
Ann. of Math. (2) {\textbf{180}} (2014), no. 2, 523--571. 


\bibitem[HX]{haconxu-lcc}
C.~D.~Hacon, C.~Xu, 
Existence of log canonical closures, 
Invent. Math. {\textbf{192}} (2013), no. 1, 161--195. 




\bibitem[H]{has-ab}
K.~Hashizume, 
Remarks on the abundance conjecture, 
Proc. Japan Acad. Ser. A Math Sci. \textbf{92} (2016) no. 9, 101--106. 


\bibitem[K]{kawamata}
Y.~Kawamata, 
Variation of mixed Hodge structures and the positivity for algebraic fiber spaces, 
Adv. Stud. Pure Math. {\textbf{65}} (2015), 
Algebraic geometry in East Asia--Taipei 2011, 27--57.


\bibitem[KMM]{kmm-abundance}
S.~Keel, K.~Matsuki, J.~M\textsuperscript{c}Kernan, 
Log abundance theorem for threefolds, 
Duke Math. J. {\textbf{75}} (1994), no. 1, 99--119. 


\bibitem[KK]{kollarkovacs}
J.~Koll\'ar, S.~Kov\'acs, 
Log canonical singularities are Du Bois, 
J. Amer. Math. Soc. {\textbf{23}} (2010), no. 3, 791--813.


\bibitem[KM]{kollar-mori} 
J.~Koll\'ar, S.~Mori, 
{\em{Birational geometry of algebraic varieties}}. With 
the collaboration of C.~H.~Clemens and A.~Corti. Translated 
from the 1998 Japanese original. Cambridge 
Tracts in Mathematics, {\textbf{134}}. Cambridge University 
Press, Cambridge, 1998.


\bibitem[L]{lai}
C.~J.~Lai, Varieties fibered by good minimal models, 
Math. Ann. {\textbf{350}} (2011), no. 3, 533--547. 


\bibitem[N]{nakayama-zariski-decom}
N.~Nakayama, 
{\em Zariski-decomposition and abundance}, MSJ Mem. {\textbf{14}}, 
Mathematical Society in Japan, Tokyo, 2004. 




\end{thebibliography}
\end{document}